\documentclass[a4paper,12pt]{article}

\usepackage{authblk}
\usepackage{parskip}
\usepackage{mathtools}
\usepackage{amssymb}
\usepackage{amsmath}
\usepackage{latexsym}
\usepackage{mathrsfs}  
\usepackage{bbm}       
\usepackage{amsthm}    
\usepackage{relsize}   
\usepackage{graphicx}
\usepackage{thmtools}  
\usepackage{mathrsfs}  
\usepackage{paralist}  
\usepackage{lipsum}    
\usepackage[all]{xy}   
\usepackage{stmaryrd}  

\numberwithin{equation}{section}

\usepackage{enumitem}  
\setlist{nosep}  
\setitemize[0]{leftmargin=*}
\setenumerate[0]{leftmargin=*}
\setenumerate[1]{label={(\arabic*)}} 

\newcommand{\N}{\mathbb{N}}     
\newcommand{\R}{\mathbb{R}}     
\newcommand{\C}{\mathbb{C}}     
\newcommand{\Prob}{\mathbb{P}}  
\newcommand{\Exp}{\mathbb{E}}   
\newcommand{\goth}[1]{\mathfrak{#1}} 
\newcommand{\inner}[2]{\left\langle #1 \, , \, #2 \right\rangle} 
\newcommand{\norm}[1]{\left|\left|#1\right|\right|}              
\newcommand{\triplet}[3]{\left( #1, #2, #3 \right) }             
\newcommand{\ProbSpace}{\triplet{\Omega}{\mathscr{F}}{\Prob}}    
 
\newcommand{\abs}[1]{\left| #1 \right|}                          
\newcommand{\defeq}{\mathrel{\mathop:}=}                         
\newcommand\restr[2]{{
  \left.\kern-\nulldelimiterspace 
  #1 
  \vphantom{\big|} 
  \right|_{#2} 
  }}
  
\makeatletter
\newsavebox{\@brx}
\newcommand{\llangle}[1][]{\savebox{\@brx}{\(\m@th{#1\langle}\)}%
  \mathopen{\copy\@brx\kern-0.5\wd\@brx\usebox{\@brx}}}
\newcommand{\rrangle}[1][]{\savebox{\@brx}{\(\m@th{#1\rangle}\)}%
  \mathclose{\copy\@brx\kern-0.5\wd\@brx\usebox{\@brx}}}
\makeatother

\theoremstyle{plain} 
\newtheorem{theorem}{Theorem}[section]    
\newtheorem{proposition}[theorem]{Proposition} 
\newtheorem{corollary}[theorem]{Corollary}
\newtheorem{lemma}[theorem]{Lemma}

\theoremstyle{definition} 
\newtheorem{definition}[theorem]{Definition}
\newtheorem{example}[theorem]{Example}
\newtheorem{remark}[theorem]{Remark}

 \title{Vector-Valued Stochastic Integration With Respect to Semimartingales in the Dual of Nuclear Spaces}
\author{C. A. Fonseca-Mora}
\affil{ Centro de Investigaci\'{o}n en Matem\'{a}tica Pura y Aplicada, \\ Escuela de Matem\'{a}tica, Universidad de Costa Rica. \\
\noindent E-mail:  christianandres.fonseca@ucr.ac.cr }
\date{}

\begin{document}

 \maketitle

\abstract{In this work, we investigate a theory of stochastic integration for operator-valued processes with respect to semimartingales taking values in the dual of a nuclear space. Our construction of this particular stochastic integral relies on previous results from [Electron. J. Probab., Volume 26, paper no. 147, 2021], together with specific tools which share some common features with good integrators in finite dimensions. We investigate various properties of this stochastic integral together with applications. In particular we obtain approximations by Riemann sums results, and provide an alternative proof of  \"{U}st\"{u}nel's version of It\^{o}'s formula involving of distributions. }

\smallskip



\section{Introduction}

Many stochastic partial differential equations (SPDEs) appearing in mathematical models of evolutionary phenomena are naturally formulated in spaces of distributions, or more generally, in duals of nuclear spaces.  In some applications, the driving noise for these evolutionary random phenomena might possess jumps (see e.g. \cite{BergerMohamed:2021, DawsonFleischmannGorostiza:1986, FonsecaMora:SPDELevy, KallianpurXiong:1994, KallianpurXiong, LokkaOksendalProske:2004, LuDai:2018}).  Among these jump-type noises, one can choose for example a distribution-valued semimartingale. Properties of SPDEs with semimartingale noise in spaces of distributions have been studied by many autors (e.g.  \cite{BojdeckiGorostiza:1991,  BojdeckiGorostiza:1999, BojdeckiGorostiza:2001,   DawsonGorostiza:1990, FernandezGorostiza:1992, Ito, KallianpurPerezAbreu:1988,  KallianpurPerezAbreu:1989, PerezAbreuTudor:1992, PerezAbreuTudor:1994, Ustunel:1982-2}). In most of these works however,  the driving noise is either a Wiener process or a square integrable martingale, with only few works where a general semimartingale noise is considered  (\cite{BojdeckiGorostiza:1991, DawsonGorostiza:1990, FonsecaMora:StochInteg}) and only for linear SPDEs.
To study solutions of nonlinear SPDEs driven by general semimartingale noise, it is first necessary to develop a robust theory of stochastic integration with respect to semimartingales that take values in spaces of distributions. 

In the context of infinite dimensional spaces, one can distinguish between real-valued and vector-valued (i.e. with values in a infinite dimensional vector space) stochastic integrals. If the underlying space is Hilbert both constructions are very similar and only an augmentation by a small amount of operator theory is needed for the the construction of the vector-valued stochastic integral (see \cite{Metivier}). However, in the case where the infinite dimensional space is the dual of a nuclear space (as for example a space of distributions) important differences arise in the construction of both types stochastic integrals, being the most important the lack of a single norm defining the topology the main complication in the construction of the vector-valued stochastic integral.    


The first theory of real-valued stochastic integration with respect to semimartingales in duals of nuclear spaces was introduced  by 
A. S. \"Ust\"unel in the series of papers \cite{Ustunel:1982, Ustunel:1982-1, Ustunel:1982-2}. There, it is assumed that the nuclear space $\Phi$ is a complete bornological, reflexive nuclear space whose strong dual $\Phi'$ is complete and nuclear and the stochastic integral is defined through the concept of projective system of semimartingales and the theory of stochastic integration for Hilbert space-valued semimartingales. In the recent work \cite{FonsecaMora:StochInteg}, the author uses a tensor product approach to introduce a  theory of real-valued stochastic integration under the assumption that $\Phi$ is a complete, barreled and nuclear (no assumptions on the dual space $\Phi'$) and the integrand is assumed to be a $\Phi'$-valued cylindrical semimartingale. 

For the case of vector valued-stochastic integration, some authors introduced theories with respect to particular classes of semimartingales and under different assumptions on the nuclear space and its strong dual. We can cite for example stochastic integration with respect to Wiener processes \cite{BojdeckiJakubowski:1990, Ito}, square integrable martingales \cite{BrooksKozinski:2010, KorezliogluMartias:1998, Mitoma:1981}, and more recently stochastic integrals with respect to L\'evy processes \cite{FonsecaMora:SPDELevy}. To the best of our knowledge, no theory of vector-valued stochastic integration with respect to general semimartingales in the dual of a nuclear space currently exists. 

\textbf{Our main contribution.}
In this work we introduce a theory of vector-valued stochastic integration for operator-valued processes and with respect to semimartingales taking values in the dual of nuclear spaces. Our main motivation is to develop a theory that can be used to study solutions to stochastic ordinary and stochastic partial differential equations in duals of nuclear spaces. These applications will appear elsewhere.  

Our construction of the stochastic integral uses a regularization argument for cylindrical semimartingales and the theory of real-valued stochastic integration introduced by the author in  \cite{FonsecaMora:StochInteg}. To explain our construction,  assume that  $\Psi$ is a quasi-complete bornological nuclear space and $\Phi$ is a complete barrelled nuclear space (e.g. $\Phi'$ and $\Psi'$ can be one of the classical spaces of distributions  $\mathscr{E}'_{K}$, $\mathscr{E}'$, $\mathscr{S}'(\R^{d})$, $\mathscr{D}'(\R^{d})$). 
Using the construction in \cite{FonsecaMora:StochInteg}, if $X$ is a $\Phi'$-valued semimartingale to every $\Phi$-valued process $H$ which is predictable and bounded we can associate a real-valued semimartingale $\int \, H \, dX$ called the \emph{stochastic integral}.
The \emph{stochastic integral mapping} $I: H \mapsto \int \, H \, dX$ is always linear, but it is not clear that it is continuous from the space of $\Phi$-valued bounded predictable processes into the space $S^{0}$ of real-valued semimartingales. If it happens that the stochastic integral mapping is continuous, we will say that $X$ is a  \emph{$S^{0}$-good integrator}. 

Now, assume that $R$ is a $\mathcal{L}(\Phi',\Psi')$-valued process that is weakly bounded and weakly predictable (our basic class of integrands). We will show as part of our construction that $R$ induces a continuous and linear operator from $\Psi$ into the space of $\Phi$-valued bounded predictable processes. This way,  the composition of $R$ and the stochastic integral mapping $I$ defines a cylindrical semimartingale in $\Phi'$ which can be radonified to a $\Psi'$-valued semimartingale using the regularization theorem for semimartingales introduced in \cite{FonsecaMora:Semi}. The resulting process is the vector-valued stochastic integral $\int \, R \, dX$ of $R$ with respect to $X$. The main advantage of such a construction is that for each $\psi \in \Psi$ we have $\inner{\int \, R \, dX}{\psi}$ is an indistinguishable version of the real-valued stochastic integral $\int \, R'\psi \, dX$ (here $R'$ is the dual operator of $R$), hence the properties of the real-valued stochastic integral in  \cite{FonsecaMora:Semi}  ``can be transferred'' to the vector-valued stochastic integral. Moreover, we show that our integral can be approximated by Riemann sums. 

A fundamental requirement for the above construction of the vector stochastic integral is that the cylindrical semimartingale obtained by the composition $R \circ I$ has to be a continuous operator from $\Psi$ into $S^{0}$; therefore it is necessary that the integrand $X$ be a $S^{0}$-good integrator. 
For that reason, in this paper we carry on a dept study on the concept of $S^{0}$-good integrator by introducing alternative characterizations, giving sufficient conditions, and providing a collection of new examples. The above study complements the developments in \cite{FonsecaMora:UCPConvergence} where a topology for the space of $S^{0}$-good integrators was introduced. 

\textbf{Organization of the paper and further description of our results.}
First, in Section \ref{sectPreliminaries} we recall basic terminology on nuclear spaces and their strong duals, cylindrical and stochastic processes, and we summarize the basic properties of the real-valued stochastic integral introduced in \cite{FonsecaMora:StochInteg}.

In Section \ref{sectGoodIntegrators} we study the properties of $S^{0}$-good integrators. In Theorem \ref{theoCharactGoodIntegrator} we show that for nuclear Fr\'echet spaces or strict inductive limits of nuclear Fr\'echet spaces the definition of $S^{0}$-good integrator coincides with the usual definitions in finite dimensions, in particular, we only require continuity of the stochastic integral mapping from the space of bounded predictable processes into the space of real-valued random variables. 
In Theorem \ref{theoLocalGoodIntegrator} we show that the property of being a $S^{0}$-good integrator is preserved under localizing sequences and that $\Phi'$-valued semimartingales that possesses a version taking values in a Hilbert space are also $S^{0}$-good integrators (Theorem \ref{theoLocalHilbertSemiIsGoodIntegra}).  
Some other properties of $S^{0}$-good integrators as well as several examples are given.  

The construction and study of the properties of the vector-valued stochastic integral is carried out in Section \ref{sectStrongStochInteg}. We start in Section \ref{sectConstruStochaInteg} with the introduction of the basic class of weakly bounded and weakly predictable operator-valued integrands. In particular, we show in Proposition \ref{propIsomorStrongIntBounPredOpeLinOper} and Corollary \ref{coroEquivStrongIntegrands} that each one of our integrands defines a continuous linear operator from $\Psi$ into the space of bounded predictable $\Phi$-valued integrands. As explained above this property if fundamental for the construction of the vector-valued integral in Theorem \ref{theoConstStrongStochInteg}.  
The basic properties of the stochastic integral are given in Section \ref{subSectProperStochInteg}. 

Our next step is our construction is given in Section \ref{sectExtStochIntegral} where we carry out an extension of the vector-valued stochastic integral to locally integrable vector-valued integrands, i.e. that admits a localizing sequence that makes them weakly bounded and weakly predictable.  Examples of such integrands are given and the stochastic integral is constructed in Theorem \ref{theoStochIntegLocallyBounded}. Furthermore,  in  Section \ref{sectRiemannRepresentation} we prove that the $S^{0}$-vector valued stochastic integral can be approximated by Riemman sums (see Theorem \ref{theoRiemannRepresentation}).  

As a final step in our construction, in Section \ref{sectIntegAsAGoodIntegrator} we explore sufficient conditions for the stochastic integral to be a $S^{0}$-good integrator. As one of our main results we show that if we assume additionally that $\Psi$ possesses the bounded approximation property, then the stochastic integral is a $S^{0}$-good integrator (Theorem \ref{theoIntegralGoodIntegrator}). In particular, we conclude that in the space of tempered distributions the stochastic integral is always a $S^{0}$-good integrator (Corollary \ref{coroStocIntegTemperedIsGoodIntegrator}). 

Finally, as an application in Section  \ref{sectItoFormula}, we revisit the proof of   \"{U}st\"{u}nel's version of It\^{o}'s formula in Section II of \cite{Ustunel:1982Ito}, and obtain a full demonstration relaying on our results on $S^{0}$-good integrators and $S^{0}$-stochastic integrals.


\section{Preliminaries}\label{sectPreliminaries}

\subsection{Nuclear spaces and their strong duals}

Let $\Phi$ be a locally convex space (we will only consider vector spaces over $\R$). $\Phi$ is called  \emph{ultrabornological} if it is the inductive limit of a family of Banach  spaces. A \emph{barreled space} is a locally convex space such that every convex, balanced, absorbing and closed subset is a neighborhood of zero. For further details see \cite{Jarchow, NariciBeckenstein}.

We denote by $\Phi'$ the topological dual of $\Phi$ and by $\inner{f}{\phi}$ the canonical pairing of elements $f \in \Phi'$, $\phi \in \Phi$. Unless otherwise specified, $\Phi'$ will always be consider equipped with its \emph{strong topology}, i.e. the topology on $\Phi'$ generated by the family of semi-norms $( \eta_{B} )$, where for each $B \subseteq \Phi$ bounded, $\eta_{B}(f)=\sup \{ \abs{\inner{f}{\phi}}: \phi \in B \}$ for all $f \in \Phi'$.  

We denote by $\mathcal{L}(\Phi,\Psi)$ the linear space of all the linear and continuous operators between any two locally convex spaces (or more generally topological vector spaces) $\Phi$ and $\Psi$. Moreover, $\mathcal{L}_{b}(\Phi,\Psi)$ denotes the space $\mathcal{L}(\Phi,\Psi)$ equipped with the topology of bounded convergence.  For information of the topologies on the space $\mathcal{L}(\Phi,\Psi)$ the reader is referred to e.g. Chapter 32 in \cite{Treves}. If $R \in \mathcal{L}(\Phi,\Psi)$ we denote by $R'$ its \emph{dual operator} and recall that $R' \in \mathcal{L}(\Psi',\Phi')$.

A continuous seminorm (respectively norm) $p$ on $\Phi$ is called \emph{Hilbertian} if $p(\phi)^{2}=Q(\phi,\phi)$, for all $\phi \in \Phi$, where $Q$ is a symmetric, non-negative bilinear form (respectively inner product) on $\Phi \times \Phi$. For any given continuous seminorm $p$ on $\Phi$ let $\Phi_{p}$ be the Banach space that corresponds to the completion of the normed space $(\Phi / \mbox{ker}(p), \tilde{p})$, where $\tilde{p}(\phi+\mbox{ker}(p))=p(\phi)$ for each $\phi \in \Phi$.By $i_{p}: \Phi \rightarrow \Phi_{p}$ we denote the unique continuous linear extension of the quotient map $\Phi \rightarrow \Phi / \mbox{ker}(p)$.  We denote by  $\Phi'_{p}$ the dual to the Banach space $\Phi_{p}$ and by $p'$ the corresponding dual norm. 
Observe that if $p$ is Hilbertian then $\Phi_{p}$ and $\Phi'_{p}$ are Hilbert spaces. If $q$ is another continuous seminorm on $\Phi$ for which $p \leq q$, we have that $\mbox{ker}(q) \subseteq \mbox{ker}(p)$ and the canonical inclusion map from $\Phi / \mbox{ker}(q)$ into $\Phi / \mbox{ker}(p)$ has a unique continuous and linear extension that we denote by $i_{p,q}:\Phi_{q} \rightarrow \Phi_{p}$. Furthermore, we have the following relation: $i_{p}=i_{p,q} \circ i_{q}$.


Let us recall that a (Hausdorff) locally convex space $(\Phi,\mathcal{T})$ is called \emph{nuclear} if its topology $\mathcal{T}$ is generated by a family $\Pi$ of Hilbertian semi-norms such that for each $p \in \Pi$ there exists $q \in \Pi$, satisfying $p \leq q$ and the canonical inclusion $i_{p,q}: \Phi_{q} \rightarrow \Phi_{p}$ is Hilbert-Schmidt. Other equivalent definitions of nuclear spaces can be found in \cite{Pietsch, Treves}. 

Let $\Phi$ be a nuclear space. If $p$ is a continuous Hilbertian semi-norm  on $\Phi$, then the Hilbert space $\Phi_{p}$ is separable (see \cite{Pietsch}, Proposition 4.4.9 and Theorem 4.4.10, p.82). Now, let $( p_{n} : n \in \N)$ be an increasing sequence of continuous Hilbertian semi-norms on $(\Phi,\mathcal{T})$. We denote by $\theta$ the locally convex topology on $\Phi$ generated by the family $( p_{n} : n \in \N)$. The topology $\theta$ is weaker than $\mathcal{T}$. We  will call $\theta$ a (weaker) \emph{countably Hilbertian topology} on $\Phi$ and we denote by $\Phi_{\theta}$ the space $(\Phi,\theta)$ and by $\widehat{\Phi}_{\theta}$ its completion. The space $\widehat{\Phi}_{\theta}$ is a (not necessarily Hausdorff) separable, complete, pseudo-metrizable (hence Baire and ultrabornological; see Example 13.2.8(b) and Theorem 13.2.12 in \cite{NariciBeckenstein}) locally convex space and its dual space satisfies $(\widehat{\Phi}_{\theta})'=(\Phi_{\theta})'=\bigcup_{n \in \N} \Phi'_{p_{n}}$ (see \cite{FonsecaMora:Existence}, Proposition 2.4). 

The class of complete ultrabornological (hence barrelled) nuclear spaces includes many spaces of functions widely used in analysis. Indeed, it is known (see e.g. \cite{Pietsch, Treves, Schaefer}) that the spaces of test functions $\mathscr{E}_{K} \defeq \mathcal{C}^{\infty}(K)$ ($K$: compact subset of $\R^{d}$), $\mathscr{E}\defeq \mathcal{C}^{\infty}(\R^{d})$, the rapidly decreasing functions $\mathscr{S}(\R^{d})$, and the space of harmonic functions $\mathcal{H}(U)$ ($U$: open subset of $\R^{d}$),  are all  examples of Fr\'{e}chet nuclear spaces. Their (strong) dual spaces $\mathscr{E}'_{K}$, $\mathscr{E}'$, $\mathscr{S}'(\R^{d})$, $\mathcal{H}'(U)$, are also nuclear spaces.
On the other hand, the space of test functions $\mathscr{D}(U) \defeq \mathcal{C}_{c}^{\infty}(U)$ ($U$: open subset of $\R^{d}$), the space of polynomials $\mathcal{P}_{n}$ in $n$-variables, the space of real-valued sequences $\R^{\N}$ (with direct sum topology) are strict inductive limits of Fr\'{e}chet nuclear spaces (hence they are also nuclear). The space of distributions  $\mathscr{D}'(U)$  ($U$: open subset of $\R^{d}$) is also nuclear.

\subsection{Cylindrical and Stochastic Processes} \label{subSectionCylAndStocProcess}

Throughout this work we assume that $\ProbSpace$ is a complete probability space equipped with a filtration $( \mathcal{F}_{t} : t \geq 0)$ that satisfies the \emph{usual conditions}, i.e. it is right continuous and $\mathcal{F}_{0}$ contains all subsets of sets of $\mathcal{F}$ of $\Prob$-measure zero. We denote by $L^{0} \ProbSpace$ the space of equivalence classes of real-valued random variables defined on $\ProbSpace$. The space $L^{0} \ProbSpace$ will be always equipped with the topology of convergence in probability and in this case it is a complete, metrizable, topological vector space.

Denote by $\mathbb{D}$ the collection of all real-valued $(\mathcal{F}_{t})$-adapted processes with c\`{a}dl\`{a}g paths. For $z \in \mathbb{D}$, let 
$$r_{ucp}(z)= \sum_{n=1}^{\infty} 2^{-n} \Exp \left( 1 \wedge \sup_{0 \leq t \leq n} \abs{z_{t}} \right),$$
which is an F-seminorm on $\mathbb{D}$. The corresponding translation invariant metric 
$$d_{ucp}(y,z)= r_{ucp}(y-z), \quad y,z \in \mathbb{D},$$
defines the topology of  convergence in probability uniformly on compact intervals of time (abbreviated as the UCP topology) on $\mathbb{D}$. 
When equipped with the UCP topology, the space $\mathbb{D}$ is a complete, metrizable, topological vector space.

We denote by $S^{0}$  the linear space (of equivalence classes) of real-valued semimartingales. On the space $S^{0}$ we define the $F$-seminorm: 
$$r_{em}(z) = \sup\{ r_{ucp}( h \cdot z) : h \in \mathcal{E}_{1}  \}, \quad z \in S^{0}$$
where  $\mathcal{E}_{1}$ is the collection of all the  real-valued predictable processes of the form 
$ \displaystyle{h=a_{0} \mathbbm{1}_{\{0\}} + \sum_{i=1}^{n-1} a_{i} \mathbbm{1}_{(t_{i}, t_{i+1}]}}$,  
for $0 < t_{1} < t_{2} < \dots < t_{n} < \infty$,   $a_{i}$ is a bounded $\mathcal{F}_{t_{i}}$-measurable random variable for $i=1, \dots, n-1$, $\abs{h} \leq 1$, and $ (h \cdot z)_{t}= a_{0} z_{0}+\sum_{i=1}^{n-1} a_{i} \left( z_{t_{i+1} \wedge t}-z_{t_{i} \wedge t} \right)$.  

The semimartingale topology on $S^{0}$ (also known as the \'Emery topology) is the topology defined by the translation invariant metric: 
$$d_{em}(y,z)= r_{em}(y-z), \quad y,z \in S^{0},$$
We always consider $S^{0}$ equipped with the semimartingale topology which makes it a complete, metrizable, topological vector space. For further details on the semimartingale topology see e.g.  Section 12.4 in \cite{CohenElliott} or Section 4.9 in \cite{KarandikarRao}.  


In this work we will make reference to several spaces of particular classes of semimartingales which we detail as follows. We denote by $\mathcal{M}_{loc}$ and $\mathscr{V}$ the subspaces of real-valued local martingales and of finite variation process.  
By $S^{c}$ we denote the subspace of $S^{0}$ of all the continuous semimartingales and by $\mathcal{M}^{c}_{loc}$ the space of continuous local martingales, both are equipped with the topology of uniform convergence in probability on compact intervals of time. Likewise $\mathcal{A}_{loc}$  denotes the space of all predictable processes of finite variation, with locally integrable variation, equipped with the F-seminorm: $\norm{a}_{\mathcal{A}_{loc}}=\Exp \left( 1 \wedge \int_{0}^{\infty} \abs{d a_{s}} \right)$. The spaces $S^{c}$, $\mathcal{M}^{c}_{loc}$ and $\mathcal{A}_{loc}$ are all closed subspaces of $S^{0}$ and the subspace topology on $S^{c}$, $\mathcal{M}^{c}_{loc}$ and $\mathcal{A}_{loc}$ coincides with their given topology (see \cite{Memin:1980}, Th\'{e}or\`{e}me IV.5 and IV.7). 
 
For every real-valued semimartingale $x=(x_{t}:t \geq 0)$ and each $p \in [1,\infty]$, we denote by $\norm{x}_{\mathcal{H}^{p}_{S}}$ the following quantity: 
$$\norm{x}_{\mathcal{H}^{p}_{S}} = \inf \left\{ \norm{ [m,m]_{\infty}^{1/2} + \int_{0}^{\infty} \abs{d a_{s} } }_{L^{p}\ProbSpace} : x=m+a \right\}, $$
where the infimum is taken over all the decompositions $x=m+a$ as a sum of a local martingale $m$ and a process of finite variation $a$. Recall that  $([m,m]_{t}: t \geq 0)$ denotes the quadratic variation process associated to the local martingale $m$, i.e. $[m,m]_{t}=\llangle  m^{c} , m^{c}  \rrangle_{t}+\sum_{0 \leq s \leq t} (\Delta m_{s})^{2}$, where $m^{c}$ is the (unique) continuous local martingale part of $m$ and $(\llangle  m^{c} , m^{c}  \rrangle_{t}:t \geq 0)$ its angle bracket process (see Section I in \cite{JacodShiryaev}). The set of all semimartingales $x$ for which $\norm{x}_{\mathcal{H}^{p}_{S}}< \infty$ is a Banach space under the norm $\norm{\cdot}_{\mathcal{H}^{p}_{S}}$ and is denoted by $\mathcal{H}^{p}_{S}$ (see Section 16.2 in \cite{CohenElliott}). Furthermore, if $x=m+a$ is a decomposition of $x$ such that $\norm{x}_{\mathcal{H}^{p}_{S}}< \infty$ it is known that in such a case $a$ is of integrable variation (see VII.98(c) in \cite{DellacherieMeyer}). 


For $p \geq 1$, denote by $\mathcal{M}_{\infty}^{p}$ the space of real-valued martingales for which $\norm{m}_{\mathcal{M}_{\infty}^{p}}= \norm{ \sup_{t \geq 0} \abs{m_{t}} }_{L^{p}\ProbSpace}< \infty$. It is well-known that $\mathcal{M}_{\infty}^{p}$ equipped with the norm $\norm{\cdot}_{\mathcal{M}_{\infty}^{p}}$ is a Banach space. Likewise, we denote by $\mathcal{A}$ the space of all predictable processes of finite variation, with integrable variation. It is well-know that $\mathcal{A}$ is a Banach space when equipped with the norm $\norm{a}_{\mathcal{A}}=\Exp \int_{0}^{t} \abs{d a_{s}} < \infty$. 


Let $\Phi$ be a locally convex space. A \emph{cylindrical random variable}\index{cylindrical random variable} in $\Phi'$ is a linear map $X: \Phi \rightarrow L^{0} \ProbSpace$ (see \cite{FonsecaMora:Existence}). If $X$ is a cylindrical random variable in $\Phi'$, we say that $X$ is \emph{$n$-integrable} ($n \in \N$)  if $ \Exp \left( \abs{X(\phi)}^{n} \right)< \infty$, $\forall \, \phi \in \Phi$. We say that $X$ has \emph{zero-mean} if $ \Exp \left( X(\phi) \right)=0$, $\forall \phi \in \Phi$. 
The \emph{Fourier transform} of $X$ is the map from $\Phi$ into $\C$ given by $\phi \mapsto \Exp ( e^{i X(\phi)})$.

Let $X$ be a $\Phi'$-valued random variable, i.e. $X:\Omega \rightarrow \Phi'$ is a $\mathscr{F}/\mathcal{B}(\Phi')$-measurable map. For each $\phi \in \Phi$ we denote by $\inner{X}{\phi}$ the real-valued random variable defined by $\inner{X}{\phi}(\omega) \defeq \inner{X(\omega)}{\phi}$, for all $\omega \in \Omega$. The linear mapping $\phi \mapsto \inner{X}{\phi}$ is called the \emph{cylindrical random variable induced/defined by} $X$. We will say that a $\Phi'$-valued random variable $X$ is \emph{$n$-integrable} ($n \in \N$) if the cylindrical random variable induced by $X$ is \emph{$n$-integrable}. 
 
Let $J=\R_{+} \defeq [0,\infty)$ or $J=[0,T]$ for  $T>0$. We say that $X=( X_{t}: t \in J)$ is a \emph{cylindrical process} in $\Phi'$ if $X_{t}$ is a cylindrical random variable for each $t \in J$. Clearly, any $\Phi'$-valued stochastic processes $X=( X_{t}: t \in J)$ induces/defines a cylindrical process under the prescription: $\inner{X}{\phi}=( \inner{X_{t}}{\phi}: t \in J)$, for each $\phi \in \Phi$. 

If $X$ is a cylindrical random variable in $\Phi'$, a $\Phi'$-valued random variable $Y$ is called a \emph{version} of $X$ if for every $\phi \in \Phi$, $X(\phi)=\inner{Y}{\phi}$ $\Prob$-a.e. A $\Phi'$-valued process $Y=(Y_{t}:t \in J)$ is said to be a $\Phi'$-valued \emph{version} of the cylindrical process $X=(X_{t}: t \in J)$ on $\Phi'$ if for each $t \in J$, $Y_{t}$ is a $\Phi'$-valued version of $X_{t}$.

A $\Phi'$-valued process $X=( X_{t}: t \in J)$ is  \emph{continuous} (respectively \emph{c\`{a}dl\`{a}g}) if for $\Prob$-a.e. $\omega \in \Omega$, the \emph{sample paths} $t \mapsto X_{t}(\omega) \in \Phi'$ of $X$ are continuous (respectively c\`{a}dl\`{a}g). 


A $\Phi'$-valued random variable $X$ is called \emph{regular} if there exists a weaker countably Hilbertian topology $\theta$ on $\Phi$ such that $\Prob( \omega: X(\omega) \in (\widehat{\Phi}_{\theta})')=1$. If $\Phi$ is a barrelled (e.g. ultrabornological) nuclear space, the property of being regular is  equivalent to the property that the law of $X$ be a Radon measure on $\Phi'$ (see Theorem 2.10 in \cite{FonsecaMora:Existence}).  A $\Phi'$-valued process $X=(X_{t}:t \geq 0)$ is said to be \emph{regular} if for each $t \geq 0$, $X_{t}$ is a regular random variable.

A \emph{cylindrical semimartingale} in $\Phi'$ is a cylindrical process $X=(X_{t}: t \geq 0)$ in $\Phi'$ such that $\forall \phi \in \Phi$, $X(\phi)$ is a real-valued semimartingale. A $\Phi'$-valued process $Y=(Y_{t}: t \geq 0)$ is called a \emph{semimartingale} if it induces a cylindrical semimartingale. We denote by $S^{0}(\Phi')$ the collection of all the $\Phi'$-valued regular c\`{a}dl\`{a}g $(\mathcal{F}_{t})$-semimartingales.

\begin{remark}\label{remarkEqualInSemimartingale}
Any two elements $X=(X_{t}: t \geq 0)$, $Y=(Y_{t}: t \geq 0) \in S^{0}(\Phi')$ are equal if they are indistinguishable. By Proposition 2.12 in \cite{FonsecaMora:Existence} a sufficient condition for $X=Y$ in $S^{0}(\Phi')$ is that for every $\phi \in \Phi$, $\inner{X_{t}}{\phi} = \inner{Y_{t}}{\phi}$ $\Prob$-a.e., i.e. if $\inner{X}{\phi}$ is a version of $\inner{Y}{\phi}$. 
\end{remark}

In general, if $\goth{S}$ denotes any space of a particular class of semimartingales (as described above), then by a $\goth{S}$-semimartingale in $\Phi'$ we mean a $\Phi'$-valued  process $X=(X_{t}: t \geq 0)$ such that $\forall \phi \in \Phi$, $\inner{X}{\phi} \in \goth{S}$.


We recall the definition of continuous part of a semimartingale. Let $X$ be a $\Phi'$-valued $(\mathcal{F}_{t})$-adapted, c\`{a}dl\`{a}g semimartingale for which the mapping $X:\Phi \rightarrow S^{0}$ is continuous from $\Phi$ into $S^{0}$. By Theorem 4.2 and Remark 4.6 in \cite{FonsecaMora:Semi} there exist a unique $\Phi'$-valued regular continuous local martingale $X^{c}=(X^{c}_{t}: t \geq 0)$ with the following property: for every $\phi \in \Phi$, if $X^{c,\phi}$ denotes the real-valued continuous local martingale corresponding to the canonical representation of $\inner{X}{\phi}$ (see VIII.45 in \cite{DellacherieMeyer}), then  the real-valued processes $\inner{X^{c}}{\phi}$ and $X^{c,\phi}$ are indistinguishable. The process $X^{c}$ is the continuous local martingale part of $X$.

\subsection{Real-valued stochastic integration}\label{sectRealValuedStochIntegration}

In this section we review some results from the theory of stochastic integration developed in Sections 4 and 5 in \cite{FonsecaMora:StochInteg} with respect to a semimartingale taking values in the dual of a nuclear space. Before we need some terminology.   

We denote by $(S^{0})_{lcx}$  the convexification of $S^{0}$, i.e. the linear space $S^{0}$ equipped with the strongest locally convex topology on $S^{0}$ that is weaker than the semimartingale topology. Since the semimartingale topology is not locally convex, the convexified topology on $S^{0}$ is strictly weaker than the semimartingale topology. 

Denote by $b\mathcal{P}$ the Banach space of all the bounded predictable processes $h : \R_{+} \times \Omega \rightarrow \R$   equipped with the uniform norm $\norm{h}_{u}=\sup_{(r,\omega)} \abs{h(r,\omega)}$. If $h \in b\mathcal{P}$ and $z \in S^{0}$, then $h$ is stochastically integrable with respect to $z$, and its stochastic integral, that we denote by $h \cdot z=(( h \cdot z)_{t}: t \geq 0)$, is an element of $S^{0}$ (see \cite{Protter}, Theorem IV.15). The mapping $(z,h) \mapsto h \cdot z$ from $S^{0} \times b\mathcal{P}$ into $S^{0}$ is bilinear (see \cite{Protter}, Theorem IV.16-7) and separately continuous (see Theorems 12.4.10-13 in \cite{CohenElliott}).  

Let  $\Phi$ be a complete barrelled nuclear space. 
We denote by $b\mathcal{P}(\Phi)$ the space of all $\Phi$-valued processes $H=(H(t,\omega): t \geq 0, \omega \in \Omega)$ with the property that $\inner{f}{H}\defeq \{ \inner{f}{H(t,\omega)}: t \geq 0, \omega \in \Omega   \} \in b\mathcal{P}$ for every $f \in \Phi'$. The space $b\mathcal{P}(\Phi)$ is complete (and Fr\'{e}chet if $\Phi$ is so)  when equipped with the topology generated by the seminorms $H \mapsto \sup_{(t,\omega)} p(H(t,\omega))$ where $p$ ranges over a generating family of seminorms for the topology on $\Phi$ (see Section 4.2 in \cite{FonsecaMora:StochInteg}). Recall that a $\Phi$-valued process is called \emph{elementary} if it takes the form  
\begin{equation}\label{eqElemenProcess}
H(t,\omega)=\sum_{k=1}^{m} h_{k}(t,\omega) \phi_{k},
\end{equation}
where for $k=1, \cdots, m$ we have $h_{k} \in b\mathcal{P}$ and $\phi_{k} \in \Phi$. By Corollary 4.9 in \cite{FonsecaMora:StochInteg} the collection of all the $\Phi$-valued elementary process is dense in $b\mathcal{P}(\Phi)$.  

Let $X=(X_{t}: t \geq 0)$ be a $\Phi'$-valued $(\mathcal{F}_{t})$-adapted semimartingale for which the mapping $\phi \mapsto X(\phi)$ is continuous from $\Phi$ into $S^{0}$. By Theorem 3.7 and Proposition 3.14 in \cite{FonsecaMora:Semi} $X$ has a regular c\`{a}dl\`{a}g version. Hence $X \in S^{0}(\Phi')$.  

Now the stochastic integral with respect to $X$ is defined as follows: by Theorem 4.10 in \cite{FonsecaMora:StochInteg} for each $H \in b\mathcal{P}(\Phi)$ there exists a real-valued c\`{a}dl\`{a}g $(\mathcal{F}_{t})$-adapted semimartingale $\int \, H \, dX$, called the \emph{stochastic integral} of $H$ with respect to $X$, such that:
\begin{enumerate}[label=(SI{\arabic*})]
\item For every $\Phi$-valued elementary process of the form \eqref{eqElemenProcess} we have 
\begin{equation}\label{eqActionWeakIntegSimpleIntegNuclear}
  \int \, H \, dX =  \sum_{k=1}^{n} \, h_{k} \cdot \inner{X}{\phi_{k}}. 
\end{equation}
\item \label{continuityStochIntegConvexif} The mapping $H \mapsto \int \, H \, dX$ is continuous from  $b\mathcal{P}(\Phi)$ into $(S^{0})_{lcx}$. 
\item \label{properBilinearity} The mapping $(H,X) \mapsto \int \, H \, dX$ is bilinear. 
\item \label{eqWeakIntegContPart} $\displaystyle{\left(\int H \, dX \right)^{c}=\int H \, dX^{c}}$.
\item \label{eqWeakIntegStopping} $\displaystyle{\left(\int H \, dX \right)^{\tau}=\int H \mathbbm{1}_{[0,\tau]} \, dX= \int H \, dX^{\tau}}$, for every stopping time $\tau$.
\end{enumerate}

We will say that a mapping $H: \R_{+} \times \Omega \rightarrow \Phi$ is locally bounded if there exists a sequence of stopping times increasing to $\infty$ $\Prob$-a.e. such that for each $n \in \N$, the mapping $(t, \omega) \mapsto H^{\tau_{n}}(t,\omega) \defeq H(t \wedge \tau_{n}, \omega)$ takes its values in a bounded subset of $\Phi$ for $\Prob$-almost all $\omega \in \Omega$.  $(\tau_{n}: n \in \N)$ is called a \emph{localizing sequence} for $H$. We denote by $\mathcal{P}_{loc}(\Phi)$ the space of all (equivalence classes of) mappings $H: \R_{+} \times \Omega \rightarrow \Phi$ that are weakly predictable and locally bounded. 
Notice that $b\mathcal{P}(\Phi) \subseteq \mathcal{P}_{loc}(\Phi)$. One can show (Theorem 5.8 in \cite{FonsecaMora:StochInteg}) that for every $H \in \mathcal{P}_{loc}(\Phi)$  there exists a real-valued c\`{a}dl\`{a}g $(\mathcal{F}_{t})$-adapted semimartingale $\int \, H \, dX$ satisfying \ref{properBilinearity}-\ref{eqWeakIntegStopping} above. 

\section{$S^{0}$-good integrators}\label{sectGoodIntegrators}

Let  $\Phi$ be a complete barrelled nuclear space. We begin with the following concept  introduced in  \cite{FonsecaMora:Semi}.

\begin{definition}
A $\Phi'$-valued $(\mathcal{F}_{t})$-adapted semimartingale $X=(X_{t}: t \geq 0)$ is a \emph{$S^{0}$-good integrator} if the mapping $\phi \mapsto X(\phi)$ is continuous from $\Phi$ into $S^{0}$ and if the  stochastic integral mapping $H \mapsto \int \, H \, dX$  defines a continuous linear mapping from $b\mathcal{P}(\Phi)$ into $S^{0}$.
\end{definition}

As it is shown in \cite{FonsecaMora:Semi}, for a $S^{0}$-good integrator the stochastic integral possesses further properties as are a Riemann representation, a stochastic integration by parts formula and a stochastic Fubini theorem (see Sections 5.2 and 6.1 in  \cite{FonsecaMora:Semi}). Moreover, for our construction of the vector stochastic integral in Section \ref{sectConstruStochaInteg} we will require our semimartigales to be $S^{0}$-good integrators. For these reasons, in this section we deepen into the study of sufficient conditions to be a $S^{0}$-good integrator and introduce several examples. 

Observe that the property of being a $S^{0}$-good integrator is not a direct consequence of property \ref{continuityStochIntegConvexif} of the stochastic integral. This because the convexified topology is strictly weaker than semimartingale's topology on $S^{0}$.  

If $X$ is an $(\mathcal{F}_{t})$-adapted semimartingale in $\Phi'$ for which the mapping $X:\Phi \rightarrow S^{0}$ is continuous from $\Phi$ into $S^{0}$, then we know by Proposition 4.12 in \cite{FonsecaMora:StochInteg} that $X$ is a $S^{0}$-good integrator in any of the following cases: 
\begin{enumerate}
\item If $X$ is a $\mathcal{H}^{p}_{S}$-semimartingale.  
\item If $X$ is a $\mathcal{M}_{\infty}^{2}$-martingale. 
\item If $X$ is a $\mathcal{A}$-semimartingale
\end{enumerate}
See also Corollary 4.13 and Proposition 7.3 in \cite{FonsecaMora:StochInteg} for other examples of $S^{0}$-good integrators.

We denote by $\mathbbm{S}^{0}(\Phi')$ the collection of all the $\Phi'$-valued semimartingales which are $S^{0}$-good integrators. It follows from \ref{properBilinearity} that $\mathbbm{S}^{0}(\Phi')$ is a linear subspace of $S^{0}(\Phi')$. 

If $\Phi$ is either a Fr\'{e}chet nuclear space or the strict inductive limit of Fr\'{e}chet nuclear spaces one can introduce a topology on $\mathbbm{S}^{0}(\Phi')$ such that the real-valued stochastic integral mapping is continuous on the integrators (see \cite{FonsecaMora:UCPConvergence}). 



 The following result shows that under some additional assumptions on $\Phi$ our definition of $S^{0}$-good integrator coincides with that in finite dimensions (e.g. in \cite{Protter}).

\begin{theorem}\label{theoCharactGoodIntegrator}
Assume $\Phi$ is either a Fr\'{e}chet nuclear space or the strict inductive limit of Fr\'{e}chet nuclear spaces. Let $X$ be a $(\mathcal{F}_{t})$-adapted semimartingale in $\Phi'$ for which the mapping $X:\Phi \rightarrow S^{0}$ is continuous from $\Phi$ into $S^{0}$. The following assertions are equivalent: 
\begin{enumerate}
\item \label{goodIntegCharac} The stochastic integral mapping $H \mapsto \int H dX$ is continuous from $b\mathcal{P}(\Phi)$ into $S^{0}$, i.e. $X$ is a $S^{0}$-good integrator. 
\item \label{ucpGoodIntegraCharac}  The stochastic integral mapping $H \mapsto \int H dX$ is continuous from $b\mathcal{P}(\Phi)$ into the space $(\mathbb{D}, ucp)$.
\item \label{weakGoodIntegraCharac} For every $t \geq 0$ the mapping $H \mapsto \int_{0}^{t} H(r) dX_{r}$ is continuous from $b\mathcal{P}(\Phi)$  into $L^{0}(\Omega, \mathcal{F}, \Prob)$.  
\end{enumerate} 
\end{theorem}
\begin{proof}
That \ref{goodIntegCharac} implies \ref{ucpGoodIntegraCharac} follows because the inclusion mapping from $S^{0}$ into $(\mathbb{D}, ucp)$ is linear and continuous. That \ref{ucpGoodIntegraCharac} implies \ref{weakGoodIntegraCharac} immediate from the definition of the UCP topology. 

We must show that \ref{weakGoodIntegraCharac} implies \ref{goodIntegCharac}. 
Let $I$ denotes the stochastic integral mapping $I(H)=\int H dX$. We already now that $I(b\mathcal{P}(\Phi)) \subseteq S^{0}$. Since $b\mathcal{P}(\Phi)$ is either a Fr\'echet space or a strict inductive limit of Fr\'{e}chet nuclear spaces (see Sections 4.2 and 4.3 in \cite{FonsecaMora:UCPConvergence}; hence $b\mathcal{P}(\Phi)$ is  ultrabornological), and $S^{0}$ is a complete metrizable topological vector space, by the closed graph theorem (see \cite{Jarchow}, Theorem 5.4.1, p.92) it suffices to show that $I$  is closed from $b\mathcal{P}(\Phi)$ into $S^{0}$. 

Let $(H_{\lambda})$ be a net converging to $H$ in $b\mathcal{P}(\Phi)$,  and let $Y \in S^{0}$ such that $I(H_{\lambda}) \rightarrow Y$ in $S^{0}$. Since convergence in the semimartingale topology implies convergence in UCP, then for each $t \geq 0$ we have $I(H_{\lambda})_{t} \rightarrow Y_{t}$ in probability. 
\
On the other hand, by our assumption we have   $I(H_{\lambda})_{t} \rightarrow I(H)_{t}$ in probability. By uniqueness of limits we have $Y_{t}=I(H)_{t}$ $\Prob$-a.e. Since both $Y$ and $I(H)$ are c\`{a}dl\`{a}g processes, then $Y$ and $I(H)$ are indistinguishable, i.e. $Y = I(H)$ in $S^{0}$. Therefore the mapping $I$ is closed, hence continuous, and $X$ is a $S^{0}$-good integrator in $\Phi'$.   
\end{proof}

Now we explore stability of the property of being a $S^{0}$-good integrator under taking the continuous part and under a stopping time. 

\begin{proposition} \label{propGoodIntegraContPartStopping}
Let $X$ be a $S^{0}$-good integrator in $\Phi'$ and let $\tau$ be a stopping time. Then $X^{c}$ and $X^{\tau}$ are $S^{0}$-good integrators. 
\end{proposition}
\begin{proof}
By definition we have that $X^{c}$ and $X^{\tau}$ are $\Phi'$-valued semimartingales. To show they are $S^{0}$-good integrators, observe that by \ref{eqWeakIntegContPart} and \ref{eqWeakIntegStopping} in Section \ref{sectRealValuedStochIntegration} for every $H \in b\mathcal{P}(\Phi)$ we have (as elements in $S^{0}$):
\begin{equation}\label{eqContinuPartAndStoppingWeakIntegral}
\left(\int H \, dX \right)^{c}= \int H \, dX^{c}, \quad 
\left(\int H \, dX \right)^{\tau}= \int H \, dX^{\tau}.
\end{equation}  
Since the mapping $H \mapsto \int H \, dX$ is continuous from $b\mathcal{P}(\Phi)$ into $S^{0}$, and the operations $z \mapsto z^{c}$ and $z \mapsto z^{\tau}$ are continuous from $S^{0}$ into $S^{0}$ (see \cite{Emery:1979}), then by \eqref{eqContinuPartAndStoppingWeakIntegral} the mappings $H \mapsto \int H \, dX^{c}$ and $H \mapsto \int H \, dX^{\tau}$ are continuous from $b\mathcal{P}(\Phi)$ into $S^{0}$.  Then $X^{c}$ and $X^{\tau}$ are $S^{0}$-good integrators. 
\end{proof}

The  next result shows that being a $S^{0}$-good integrator is a `local' property.  

\begin{theorem}\label{theoLocalGoodIntegrator}
Assume that  $\Phi$ is a complete bornological barrelled nuclear space (e.g. if $\Phi$ is a complete ultrabornological nuclear space). Let $X$ be a $\Phi'$-valued $(\mathcal{F}_{t})$-adapted process for which  there exist an increasing sequence  $(\tau_{n})$  of stopping times such that $\tau_{n} \rightarrow \infty$ $\Prob$-a.e. and a sequence  $(X^{n})$ of $S^{0}$-good integrators in $\Phi'$ such that for each $n \in \N$, $X^{\tau_{n}}=(X^{n})^{\tau_{n}}$. Then $X$  is a $S^{0}$-good integrator. 
\end{theorem}
\begin{proof}
Given $\phi \in \Phi$, we have
\begin{equation}\label{eqSemimarSequenceStopping}
\inner{X^{\tau_{n}}}{\phi}=\inner{(X^n)^{\tau_{n}}}{\phi}=\inner{X^{n}}{\phi}^{\tau_{n}} \in S^{0}. 
\end{equation}
Hence by Theorem II.2.6 in \cite{Protter}, p.54, we have $\inner{X}{\phi} \in S^{0}$. Therefore $X$ is a $\Phi'$-valued semimartingale. 

Now we show that the mapping $X:\Phi \rightarrow S^{0}$ is continuous. Being $\Phi$ ultrabornological it suffices to show that this mapping is sequentially continuous
(see Theorem 2.1 and Proposition 4.1 in \cite{FerrerMoralesSanchezRuiz})

Assume that $\phi_{k} \rightarrow \phi$ in $\Phi$. Since the operation $z \mapsto z^{\tau_{n}}$ is continuous from $S^{0}$ into $S^{0}$ (see \cite{Emery:1979}), then by \eqref{eqSemimarSequenceStopping} for every $n \in \N$ we have $\inner{X^{\tau_{n}}}{\phi_{k}} \rightarrow \inner{X^{\tau_{n}}}{\phi}$ in $S^{0}$. Because  $(\tau_{n})$  is an increasing sequence of stopping times such that $\tau_{n} \rightarrow \infty$ $\Prob$-a.e. by Lemma 12.4.8 in \cite{CohenElliott}, p.279, we have  $\inner{X}{\phi_{k}} \rightarrow \inner{X}{\phi}$ in $S^{0}$.
Hence the mapping $X:\Phi \rightarrow S^{0}$ is sequentially continuous, therefore continuous. Moreover, by Theorem 3.7 and Proposition 3.14 in \cite{FonsecaMora:Semi} $X$ has a regular c\`{a}dl\`{a}g version.

To conclude that $X$ is a $S^{0}$-good integrator we must show that the stochastic integral mapping $H \mapsto \int H dX$ is continuous from $b\mathcal{P}(\Phi)$ into $S^{0}$. As before, since $\Phi$ is bornological it suffices to show the stochastic integral mapping is sequentially continuous. 

Let $H_{k} \rightarrow H$ in $b\mathcal{P}(\Phi)$. For each $k,n \in \N$, we have
$$ \left( \int H_{k} dX \right)^{\tau_{n}} = \int H_{k} dX^{\tau_{n}} = \int H_{k} d (X^{n})^{\tau_{n}}, $$
and similarly for $H$. Now because $X^{n}$ is a $S^{0}$-good integrator, by Proposition \ref{propGoodIntegraContPartStopping} $(X^{n})^{\tau_{n}}$ is also a $S^{0}$-good integrator. Then, since $H_{k} \rightarrow H$, for each $n \in \N$ we have
$$ d_{S^{0}}\left( \left( \int H_{k} dX \right)^{\tau_{n}} - \left( \int H dX \right)^{\tau_{n}} \right) = d_{S^{0}}\left( \int H_{k} d(X^{n})^{\tau_{n}} - \int H d(X^{n}) ^{\tau_{n}} \right) \rightarrow 0, $$
as $k \rightarrow \infty$.   Since $(\tau_{n})$  is an increasing sequence of stopping times such that $\tau_{n} \rightarrow \infty$ $\Prob$-a.e. by Lemma 12.4.8 in \cite{CohenElliott}, p.279, we have $\int H_{k} dX \rightarrow \int H dX$ in $S^{0}$. This shows the stochastic integral mapping associated to $X$ is sequentially continuous, therefore continuous. Thus $X$ is a $S^{0}$-good integrator. 
\end{proof}

As the usual practice, we say that a \emph{property $\pi$ hold locally} for a $\Phi'$-valued adapted process $X$ if  there exists an increasing sequence  $(\tau_{n})$  of stopping times such that $\tau_{n} \rightarrow \infty$ $\Prob$-a.e. and  $X^{\tau_{n}}$ has property $\pi$ for each $n \in \N$. 

As a direct consequence of Theorem \ref{theoLocalGoodIntegrator} we obtain the following:

\begin{corollary}\label{coroLocalGoodIntegrator}
Assume that  $\Phi$ is a complete bornological barrelled nuclear space. Let $X$ be an $(\mathcal{F}_{t})$-adapted process which is locally a $S^{0}$-good integrator in $\Phi'$.  Then $X$  is a $S^{0}$-good integrator. 
\end{corollary}

The result in Corollary \ref{coroLocalGoodIntegrator} is of great application for concrete examples as it reduces the problem of checking that a given process is a $S^{0}$-good integrator to check that this property hold locally. Since we already know some important examples of $S^{0}$-good integrators, we obtain the following:

\begin{corollary}\label{coroConcreteLocalGoodIntegrators} Assume that  $\Phi$ is a complete bornological barrelled nuclear space. In each of the following situations a $\Phi'$-valued $(\mathcal{F}_{t})$-adapted process $X$ is a $S^{0}$-good integrator:
\begin{enumerate}
\item If $X$ is locally a $\mathcal{H}^{p}_{S}$-semimartingale.  
\item If $X$ is locally a $\mathcal{M}_{\infty}^{2}$-martingale. 
\item If $X$ is locally a $\mathcal{A}$-semimartingale.
\end{enumerate}
\end{corollary}

\begin{example}\label{examSquareInteMartGoodIntegr}
 Let $\Phi$ be a complete bornological barrelled nuclear space.   A $\Phi'$-valued \emph{square integrable martingale} is a  $(\mathcal{F}_{t})$-adapted process $M=(M_{t}: t \geq 0)$ such that for each $\phi \in \Phi$ we have $\inner{M}{\phi}$ is a real-valued square integrable martingale. Clearly, any $\mathcal{M}_{\infty}^{2}$-martingale is a  $\Phi'$-valued square integrable martingale but the converse does not hold in general. 

It is clear that $M$ being a $\Phi'$-valued square integrable martingale is a $\Phi'$-valued $(\mathcal{F}_{t})$-adapted semimartingale. We will show that if for each $t \geq 0$ the random variable $M_{t}$ has a Radon probability distribution, then $M$ is a $S^{0}$-good integrator. 

In effect, if  $M_{t}$ has a Radon distribution by Theorem 2.10 in \cite{FonsecaMora:Existence}   the  mapping $M_{t}: \Phi \rightarrow L^{0}(\Omega, \mathcal{F}, \Prob)$ is continuous. Then by the arguments used in the proof of Theorem  5.2 in \cite{FonsecaMora:Existence} for every $T>0$ the family $(M_{t}: t \in [0,T])$ is equicontinuous from $\Phi$ into $L^{0}(\Omega, \mathcal{F}, \Prob)$, and by 
Proposition 3.14 in \cite{FonsecaMora:Semi} the mapping $M:\Phi \rightarrow S^{0}$ is continuous. Hence, if we consider an increasing sequence of positive real-valued numbers $(T_{n})$ such that $T_{n} \rightarrow \infty$, then for each $n \in \N$ we have $M^{T_{n}}$ is a $\mathcal{M}_{\infty}^{2}$-martingale. Therefore, $M$ is locally a $\mathcal{M}_{\infty}^{2}$-martingale hence a  $S^{0}$-good integrator  by Corollary \ref{coroConcreteLocalGoodIntegrators}. 

If for example $\Phi'$ is a Suslin space, then the probability distribution of each $M_{t}$ is a Radon measure on $\Phi'$ (see \cite{BogachevMT}, Theorem 7.4.3, p.85). In particular, the spaces $\mathscr{E}'$, $\mathscr{S}'(\R^{d})$, $\mathscr{D}'(\R^{d})$ are all Suslin (\cite{SchwartzRM}, p.115). Therefore, a square integrable martingale taking values in any of these spaces is a $S^{0}$-good integrator. 
In particular, this example shows that any $\mathscr{S}'(\R^{d})$-valued Wiener process (see \cite{Ito}) is a $S^{0}$-good integrator.  
\end{example}

A $\Phi'$-valued \emph{locally square integrable martingale} is a $(\mathcal{F}_{t})$-adapted process $M$ such that for each $\phi \in \Phi$ we have $\inner{M}{\phi}$ is locally a square integrable martingale (the localizing sequence depends on $\phi$). In general, it is not clear if such an $M$ is a $S^{0}$-good integrator. However, if one can find a   localizing sequence $(\tau_{n}: n \in \N)$ (not depending on $\phi$) such that for each $n \in \N$ and $\phi \in \Phi$ we have $\inner{M^{\tau_{n}}}{\phi}$ is a square integrable martingale, then by Corollary \ref{coroLocalGoodIntegrator} and Example \ref{examSquareInteMartGoodIntegr} we have $M$ is a $S^{0}$-good integrator. An example of this situation is given below.

\begin{example}
Let $m=(m_{t}: t \geq 0)$ denote a real-valued continous local martingale, i.e. $m \in \mathcal{M}^{c}_{loc}$. Let $(\tau_{n}:n \in \N)$ be a localizing sequence for $m$ such that for each $n \in \N$, $m^{\tau_{n}} \in \mathcal{M}^{2}_{\infty}$.  For every $t \geq 0$ define 
$$X_{t}(\phi)= \int_{0}^{t} \phi(s) dm_{s}, \quad \forall \phi \in \mathscr{S}(\R). $$
Using the properties of the stochastic integral we have for each $t \geq 0$ that  $X_{t}: \mathscr{S}(\R) \rightarrow \mathcal{M}^{c}_{loc}$ is continuous. Therefore, by Proposition 3.12 in \cite{FonsecaMora:Semi}  we have $X=(X_{t}: t \geq 0)$ defines a $\mathscr{S}'(\R)$-valued local martingale with continuous paths and Radon distributions. Moreover, for every $n \in \N$, we have
$$ \inner{X_{t}^{\tau_{n}}}{\phi}= \left( \int_{0}^{t} \phi(r) dm_{r} \right)^{\tau_{n}} =   \int_{0}^{t} \phi(r) dm^{\tau_{n}}_{r} \in \mathcal{M}_{\infty}^{2}.$$
Hence, $X$ is locally a $\mathcal{M}_{\infty}^{2}$-martingale. By Corollary \ref{coroConcreteLocalGoodIntegrators}, $X$ is a $S^{0}$-good integrator.  
\end{example}

In many situations we will be able to show that a $\Phi'$-valued semimartingale $X$ is locally a semimartingale in some Hilbert space $\Phi'_{p}$. As we shall see below (Theorem \ref{theoLocalHilbertSemiIsGoodIntegra}), this property implies that $X$ is a $S^{0}$-good integrator. In view of Corollary \ref{coroLocalGoodIntegrator}, the main step is to show that a $\Phi'_{p}$-valued semimartingale is a $S^{0}$-good integrator. This is proved in the next result.

\begin{lemma}\label{lemmaHilbertSemiIsGoodIntegrator} Assume $\Phi$ is either a Fr\'{e}chet nuclear space or the strict inductive limit of Fr\'{e}chet nuclear spaces.
Let $p$ be a continuous Hilbertian seminorm on $\Phi$ and assume that $X=(X_{t}: t\geq 0)$ is a $\Phi'_{p}$-valued c\`{a}dl\`{a}d adapted semimartingale. Let $Y=(Y_{t}: t\geq 0)$ given by $Y_{t}=i'_{p}X_{t}$. Then $Y$ is a $S^{0}$-good integrator.  
\end{lemma}
\begin{proof}
Since $\Phi'_{p}$ is a Hilbert space, by an application of the closed graph theorem one can conclude that the mapping $X:\Phi'_{p} \rightarrow S^{0}$, $\phi \mapsto X(\phi)=\inner{X}{\phi}$, is linear and continuous. Hence $Y_{t} = i'_{p} \circ X_{t} $ is a  $\Phi'$-valued regular c\`{a}dl\`{a}d  $(\mathcal{F}_{t})$-semimartingale and the mapping $Y:\Phi \rightarrow S^{0}$ is continuous from $\Phi$ into $S^{0}$. It remains to show that $Y$ is a $S^{0}$-good integrator.

Indeed, since $X$ is a $\Phi'_{p}$-valued semimartingale, by  Theorem 23.14 in \cite{Metivier} $X$  admits a control process $A$, i.e. $A$ is an increasing, positive, adapted process such that for every stopping time $\tau$ and every $\Phi_{p}$-valued predictable elementary process $H$, one has
\begin{equation}\label{eqControlProcessHilbertSemim}
\Exp \left[ \sup_{0 \leq t < \tau} \abs{\int_{0}^{t} H(r) dX_{r} }^{2} \right] \leq \Exp \left[ A_{\tau-}  \int_{0}^{\tau-} p(H(r))^{2} dA_{r}\right]. 
\end{equation} 

Now  if $H$ is a $\Phi$-valued predictable elementary process. Then $i_{p}H$ is a $\Phi_{p}$-valued predictable elementary process, where
$$ p(i_{p} H(t,\omega))=p(H(t,\omega)), \quad \forall t \geq 0, \omega \in \Omega. $$ 
 Observe moreover that if $H(t,\omega)=\sum_{k=1}^{n} h_{k}(t,\omega) \phi_{k}$, for $h_{k} \in b \mathcal{P}$, $\phi_{k} \in \Phi$, then 
 $$ \int H dY = \sum_{k=1}^{n} h_{k} \cdot Y(\phi_{k}) = \sum_{k=1}^{n} h_{k} \cdot X(i_{p}  \phi_{k}) = \int i_{p} H dX. $$

Let $\epsilon>0$. For every $n \in \N$, let $\tau_{n}=\inf\{ t \geq 0: \abs{A_{t}}^{2} \geq n \}$. Then  $(\tau_{n})$ is an increasing sequence of stopping times such that $\tau_{n} \rightarrow \infty$ $\Prob$-a.e. Then for each $\Phi$-valued elementary predictable process $H$ we have by \eqref{eqControlProcessHilbertSemim}  that
\begin{eqnarray*}
\Prob \left( \sup_{0 \leq t < \tau_{n}} \abs{\int_{0}^{t} H(r) dY_{r} } \geq \epsilon \right) 
& \leq & \frac{1}{\epsilon^{2}} \Exp \left[ \sup_{0 \leq t < \tau_{n}} \abs{\int_{0}^{t} i_{p} H(r) dX_{r} }^{2} \right] \\
& \leq & \frac{1}{\epsilon^{2}} \Exp \left[ A_{\tau_{n}-}  \int_{0}^{\tau_{n}-} p(i_{p} H(r))^{2} dA_{r}\right] \\
& \leq & \frac{n}{\epsilon^{2}} \sup_{(t,\omega)} p(i_{p} H(t,\omega))^{2}.  
\end{eqnarray*} 
From the above and since $\tau_{n} \rightarrow \infty$, we conclude that the mapping $H \mapsto \int H dX$ is continuous from the space of $\Phi$-valued predictable elementary processes equipped with the topology of uniform convergence on $[0,\infty) \times \Omega$ into the space $(\mathbb{D}, ucp)$. By a density argument, this mapping admits a linear and continuous extension from $b\mathcal{P}(\Phi)$ into the space $(\mathbb{D}, ucp)$. By Theorem \ref{theoCharactGoodIntegrator} we conclude that $Y$ is a $S^{0}$-good integrator. 
\end{proof}

\begin{definition}\label{defiLocallyHilbertianSemimartingale}
Assume that $X=(X_{t}: t \geq 0)$ is a $\Phi'$-valued process for which there exists  an increasing sequence  $(\gamma_{n}: n \in \N)$  of stopping times such that $\gamma_{n} \rightarrow \infty$ $\Prob$-a.e., and there exists and increasing sequence $(q_{n}: n \in \N)$ of continuous Hilbertian seminorms on $\Phi$, such that for each $n \in \N$ the process $X^{\gamma_{n}}$ possesses an indistinguishable version which is a $\Phi'_{q_{n}}$-valued c\`{a}dl\`{a}d adapted semimartingale. In the above situation we will say that $X$ is \emph{locally a Hilbertian semimartingale} in $\Phi'$ with corresponding sequence of stopping times $(\tau_{n})$ and Hilbertian seminorms $(q_{n})$. 
\end{definition}

\begin{theorem}\label{theoLocalHilbertSemiIsGoodIntegra}
Assume $\Phi$ is either a Fr\'{e}chet nuclear space or the strict inductive limit of Fr\'{e}chet nuclear spaces and let $X$ be locally a Hilbertian semimartingale in $\Phi'$ with corresponding sequence of stopping times $(\tau_{n})$ and Hilbertian seminorms $(q_{n})$.  Then $X$ is a $S^{0}$-good integrator. 
\end{theorem}
\begin{proof}
Given $n \in \N$, by Lemma \ref{lemmaHilbertSemiIsGoodIntegrator} we have $X^{\tau_{n}}$ is a $\Phi'$-valued regular c\`{a}dl\`{a}d  $(\mathcal{F}_{t})$-semimartingale which is a $S^{0}$-good integrator. Then by Corollary \ref{coroLocalGoodIntegrator} we have that $X$ is a $S^{0}$-good integrator.
\end{proof}

\begin{example}
\label{examDiracGoodIntegrator}

Let $z=(z_{t}: t \geq 0)$ be a $\R^{d}$-valued semimartingale. For $t \geq 0$ and $\omega \in \Omega$ we define a linear mapping on $\mathscr{E}(R^{d})$ by the prescription:
$$X_{t}(\omega)(\phi)=\delta_{z_{t}(\omega)}(\phi)=\phi(z_{t}(\omega)), \quad \, \forall \, \phi \in  \mathscr{E}(\R^{d}).$$ 
By It\^{o}'s formula, we have $(X_{t}(\phi): t \geq 0) \in S^{0}$ for every $ \phi \in  \mathscr{E}(\R^{d})$. Moreover, for every $t \geq 0$ the mapping $X_{t}$ is continuous from $\mathscr{E}(\R^{d})$ into $L^{0} \ProbSpace$. By Corollary 3.8 in \cite{FonsecaMora:Semi} $X=(X_{t}: t \geq 0)$ defines a $\mathscr{E}'(\R^{d})$-valued regular  c\`{a}dl\`{a}g $(\mathcal{F}_{t})$-semimartingale for which the mapping $X: \mathscr{E}(\R^{d}) \rightarrow S^{0}$ is continuous. 

We will show that $X$ is a $S^{0}$-good integrator. In effect, for every $n \in \N$, let $\tau_{n}=\inf \{ t \geq 0: \norm{z_{t}} >n  \} \wedge n$. Then, $(\tau_{n}: n \in \N)$ is a increasing sequence of stopping times such that $\tau_{n} \rightarrow \infty $ $\Prob$-a.e. 
For fixed $n \in \N$, we have
$$ X^{\tau_{n}}_{t}(\phi)= \phi(z_{t \wedge \tau_{n}}), \quad \forall t\geq 0, \, \phi \in  \mathscr{E}(\R^{d}).$$
Moreover, for every $\phi \in \mathscr{E}(\R^{d})$ we have $\Exp  \left[ \sup_{t \geq 0} \abs{X^{\tau_{n}}_{t}(\phi)}^{2} \right]< \infty$, this because
$(\phi(z_{t \wedge \tau_{n}}): t \geq 0)$ is a bounded process. In particular, the mapping $\varrho:  \mathscr{E}(\R^{d}) \rightarrow \R_{+}$ given by
$$ \varrho(\phi)= \left( \Exp  \left[ \sup_{t \geq 0} \abs{X^{\tau_{n}}_{t}(\phi)}^{2} \right] \right)^{1/2}, \quad \forall \, \phi \in \mathscr{E}(\R^{d}), $$
defines a seminorm on $\mathscr{E}(\R^{d})$. With the help of Fatou's lemma one can show $\varrho$ is sequentially lower-semicontinuous, hence continuous since $\mathscr{E}(\R^{d})$ is ultrabornological. 
Then, by Theorem 4.3 in \cite{FonsecaMora:Existence} there exists a continuous Hilbertian seminorm $p$ on $\mathscr{E}(\R^{d})$, $\varrho \leq p$,  such that  $X^{\tau_{n}}$
has an indistinguishable  $\mathscr{E}'(\R^{d})_{p}$-valued c\`{a}dl\`{a}g version $ \tilde{X}^{\tau_{n}}$ satisfying $ \Exp \left[ \sup_{t \geq 0} p'(\tilde{X}^{\tau_{n}}_{t})^{2} \right]< \infty$. Observe that this version $\tilde{X}^{\tau_{n}}$ is only a cylindrical semimartingale, however, if we choose a continuous Hilbertian seminorm $q$ on $\mathscr{E}(\R^{d})$, $p \leq q$, satisfying $i_{p,q}$ is Hilbert-Schmidt, then by Theorem A in \cite{JakubowskiEtAl:Radonification} we have $Y^{n}=i'_{p,q}\tilde{X}^{\tau_{n}}$ is a  $\mathscr{E}'(\R^{d})_{q}$-valued c\`{a}dl\`{a}g semimartingale which is an indistinguishable version of $X^{\tau_{n}}$. 

We have proved that $X$ is locally a Hilbertian semimartingale in $\mathscr{E}'(\R)$. Then by Theorem \ref{theoLocalHilbertSemiIsGoodIntegra} we have $X$ is a $S^{0}$-good integrator.    
\end{example}

\begin{definition} A $\Phi'$-valued process $X=( X_{t} : t\geq 0)$ is called a \emph{L\'{e}vy process} if 
\begin{enumerate}[label=(\roman*)]
\item  $X_{0}=0$ a.s., 
\item $X$ has \emph{independent increments}, i.e. for any $n \in \N$, $0 \leq t_{1}< t_{2} < \dots < t_{n} < \infty$ the $\Phi'$-valued random variables $X_{t_{1}},X_{t_{2}}-X_{t_{1}}, \dots, X_{t_{n}}-X_{t_{n-1}}$ are independent,  
\item L has \emph{stationary increments}, i.e. for any $0 \leq s \leq t$, $X_{t}-X_{s}$ and $X_{t-s}$ are identically distributed, and  
\item For every $t \geq 0$ the distribution $\mu_{t}$ of $X_{t}$ is a Radon measure and the mapping $t \mapsto \mu_{t}$ from $\R_{+}$ into the space $\goth{M}_{R}^{1}(\Phi')$ of Radon probability measures on $\Phi'$ is continuous at $0$ when $\goth{M}_{R}^{1}(\Phi')$  is equipped with the weak topology. 
\end{enumerate}
\end{definition}

Basic properties of L\'evy processes in duals of nuclear spaces has been investigated in \cite{FonsecaMora:Levy}. 
The next example shows some classes of nuclear Fr\'echet spaces for which every L\'evy process is a $S^{0}$-good integrator. 

\begin{example}\label{exampleLevyGoodIntegCHNS}
We consider the following class of countably Hilbertian nuclear spaces  taken from \cite{KallianpurXiong}, Example 1.3.2. 

Let $(H, \inner{\cdot}{\cdot}_{H})$ be a separable Hilbert space and $-L$ be a closed densely defined self-adjoint operator on $H$ such that $\inner{-L \phi}{\phi}_{H} \leq 0$ for each $\phi \in \mbox{Dom}(J)$.  Assume moreover that  some power of the resolvent of $L$ is a Hilbert-Schmidt operator, i.e. there exists some $\beta$ such that $(\lambda I+ L)^{-\beta}$ is Hilbert-Schmidt. Then, there exists a complete orthonormal set $(\phi_{j}: j \in \N) \subseteq H$ and a sequence $0 \leq \lambda_{1} \leq \lambda_{2} \leq \cdots $ such that  $L \phi_{j} = \lambda_{j} \phi_{j}$, $j \in \N$, and 
\begin{equation}\label{eqSumEigenvaluesExampleCHNS}
\sum_{j=1}^{\infty} (1+\lambda_{j})^{-2\beta} < \infty. 
\end{equation}
Let 
$$ \Phi \defeq \left\{ \phi \in H: \sum_{j=1}^{\infty} (1+\lambda_{j})^{2r} \inner{\phi}{\phi_{j}}^{2}_{H} < \infty, \, \forall \,  r \in \R  \right\},$$
and define an inner product $\inner{\cdot}{\cdot}_{r}$ on $\Phi$ by 
$$ \inner{\phi}{\psi}_{r} \defeq \sum_{j=1}^{\infty} (1+\lambda_{j})^{2r} \inner{\phi}{\phi_{j}}_{H} \inner{\psi}{\phi_{j}}_{H},  $$
with corresponding Hilbertian norm $\norm{\phi}_{r}^{2} \defeq \inner{\phi}{\phi}_{r}$. If $\Phi_{r}$ denotes the completion of $\Phi$ with respect to $\norm{\cdot}_{r}$, then $\Phi_{0}=H$ and $\Phi_{s} \subseteq \Phi_{r}$ for $s \leq r$ since  $\norm{\cdot}_{r} \leq \norm{\cdot}_{s}$. Furthermore,  by \eqref{eqSumEigenvaluesExampleCHNS} the canonical inclusion from $\Phi_{s}$ into $\Phi_{r}$ is Hilbert-Schmidt for $s \geq r+ \beta$. Since $\Phi=\cap_{r} \Phi_{r}$, if $\Phi$ is equipped with the topology generated by the family $(\norm{\cdot}_{r}, r \in \N)$, we have $\Phi$ is a countably Hilbertian nuclear space. 

With $\Phi$ as constructed above, let $X=(X_{t}: t \geq 0)$ be a $\Phi'$-valued L\'evy process. It is shown in (\cite{PerezAbreuRochaArteagaTudor:2005}, Theorem 4) that for every $T>0$ there exists some  $r_{T} \in \N$ such that $(X_{t}: 0 \leq t \leq T)$ has a version which is a $\Phi'_{r_{T}}$-valued L\'evy process. Therefore, by Theorem \ref{theoLocalHilbertSemiIsGoodIntegra} we have $X$ is a $S^{0}$-good integrator. 
\end{example}

\begin{remark}
It is well-known that the Schwartz space of rapidly decreasing functions $\mathscr{S}(\R)$ can be generated by following the procedure of Example \ref{exampleLevyGoodIntegCHNS} (see Remark 1.3.5 in \cite{KallianpurXiong}). As a consequence, every  $\mathscr{S}'(\R)$-valued L\'evy process is a $S^{0}$-good integrator.  
\end{remark}

As the next result shows, the property of being a $S^{0}$-good integrator is preserved under the image of a continuous linear operator. 

\begin{proposition}\label{propContinuousImageGoodIntegrator}
Let $\Phi$ and $\Psi$ be two complete barrelled nuclear spaces and let $A \in \mathcal{L}(\Phi',\Psi')$. If $X$ is a $S^{0}$-good integrator in $\Phi'$, then $Y=(A(X_{t}): t \geq 0)$ is a $S^{0}$-good integrator in $\Psi'$ and 
\begin{equation}\label{eqIntegralOfContiImageGoodIntegrator}
\int \, H \, dY = \int \, A'(H) \, dX, \quad \forall H \in b\mathcal{P}(\Psi).
\end{equation}
\end{proposition}
\begin{proof} It is clear that $Y$ is $\Psi'$-valued $(\mathcal{F}_{t})$-adapted semimartingale and the mapping $\psi \mapsto Y(\psi)$ is continuous from $\Psi$ into $S^{0}$. 

First we prove that \eqref{eqIntegralOfContiImageGoodIntegrator} holds true. Indeed, observe that if $H$ is of the elementary form 
$$H(t,\omega)=\sum_{k=1}^{m} h_{k}(t,\omega) \psi_{k},$$
where  $h_{k} \in b\mathcal{P}$ and $\psi_{k} \in \Psi$ for $k=1, \cdots, m$, then 
$$ \int \, H \, dY = \sum_{k=1}^{m} h_{k} \cdot \inner{Y}{\psi_{k}} = 
\sum_{k=1}^{m} h_{k} \cdot \inner{X}{A'\psi_{k}}= \int \, A'(H) \, dX. $$
Since the stochastic integral mapping $H \mapsto \int \, H \, dY$ (respectively $K \mapsto \int \, K \, dX$) is continuous from  $b\mathcal{P}(\Psi)$ into $(S^{0})_{lcx}$ (respectively from  $b\mathcal{P}(\Phi)$ into $S^{0}$),  and the elementary integrands are dense in $b\mathcal{P}(\Psi)$, we obtain \eqref{eqIntegralOfContiImageGoodIntegrator}. 

To conclude that $Y$ is a $S^{0}$-good integrator we must  show that the 
mapping $H \mapsto \int \, H \, dY$ is continuous from  $b\mathcal{P}(\Phi)$ into $S^{0}$. 

Assume $(H_{\gamma}: \gamma \in \Gamma)$ is a net converging to $H$ in $b\mathcal{P}(\Psi)$.  Since  $A' \in \mathcal{L}(\Psi,\Phi)$, we have $(A'(H_{\gamma}): \gamma \in \Gamma)$ converges to $A'(H)$ in $b\mathcal{P}(\Phi)$. In effect, if $p$ is a continuous seminorm on $\Phi$, by continuity of $A'$ there exist $C>0$ and a continuous seminorm $q$ on $\Psi$ such that $p(A'\psi) \leq C q(\psi)$ for every $\psi \in \Psi$. Then, we have 
$$ \sup_{(t,\omega)} p(A'(H_{\gamma}(t,\omega))-A'(H(t,\omega))) \leq C \sup_{(t,\omega)} q(H_{\gamma}(t,\omega)-H(t,\omega)) \rightarrow 0. $$
Hence, by \eqref{eqIntegralOfContiImageGoodIntegrator} and because $X$ is a $S^{0}$-good integrator we have 
$$ \int \, H_{\gamma} \, dY = \int \, A'(H_{\gamma}) \, dX \rightarrow \int \, A'(H) \, dX = \int \, H \, dY ,$$ 
which shows that $Y$ is a $S^{0}$-good integrator. 
\end{proof}

\section{Vector-Valued $S^{0}$-Stochastic Integration}\label{sectStrongStochInteg}

\subsection{Construction of the vector-valued $S^{0}$-stochastic integral}\label{sectConstruStochaInteg}

In this section, within our framework, we investigate a suitable vector-valued stochastic integral for operator-valued processes with respect to dual of a nuclear space-valued semimartingales. In particular, the stochastic integral will be constructed using a regularization argument  which employs the real-valued stochastic integral as a building block. For this argument to be employed, one needs the continuity of the  stochastic integral mapping into the space $S^{0}$ (with the semimartingale topology). Therefore, we will consider only $S^{0}$-good integrators.  

We start in this section by considering the following class of operator-valued processes:

\begin{definition}
Let $\Phi$ and $\Psi$ be locally convex spaces. Denote by $b\mathcal{P}(\Psi, \Phi)$ the space of mappings $R:\R_{+} \times \Omega \rightarrow \mathcal{L}(\Psi,\Phi)$ that are:
\begin{enumerate}
\item \emph{weakly predictable}, that is $\forall f \in \Phi'$, $\psi \in \Psi$, the mapping $(t,\omega) \mapsto \inner{f}{R(t,\omega) \psi}$ is predictable, 
\item  \emph{weakly bounded}, that is for every $B \subseteq \Psi$ bounded, 
$$\sup_{\psi \in B} \sup_{(t,\omega)} \abs{ \inner{f}{R(t,\omega) \psi}}< \infty, \quad \forall \, f  \in \Phi'.$$ \end{enumerate}
\end{definition}  

The space $b\mathcal{P}(\Psi, \Phi)$ is a linear space. To see this, let $R, S \in b\mathcal{P}(\Psi, \Phi)$ and $c \in \R$. For any $\psi \in \Psi$, $f \in \Phi'$, $t \geq 0$, $\omega \in \Omega$, we have
$$ \inner{f}{(cR(t,\omega)+S(t,\omega))\psi}= \inner{f}{cR(t,\omega) \psi}+ \inner{f}{S(t,\omega) \psi}. $$
From the above it is easy to conclude that $cR+S \in b\mathcal{P}(\Phi',\Psi')$. 

In this section we introduce a theory of  $\Psi'$-valued stochastic integration for integrands that belongs to $b\mathcal{P}(\Phi', \Psi')$ with respect to a $S^{0}$-good integrator in $\Phi'$, and under rather general conditions on $\Phi$ and $\Psi$. The first step in our construction will be to show that for reflexive spaces (we do not need to assume they are nuclear spaces) we can identify the spaces $ b\mathcal{P}(\Psi,\Phi)$ and $ b\mathcal{P}(\Phi',\Psi')$. 

\begin{proposition} \label{propDualIntegrandsNuclearSpace}
Suppose that $\Psi$ and $\Phi$ are reflexive locally convex spaces. Then, the mapping 
\begin{equation} \label{eqDefiIsomoNucleIntegrands}
\mathcal{L}(\Phi', \Psi') \ni R(t,\omega) \mapsto R(t,\omega)' \in \mathcal{L}(\Psi,\Phi), \quad \forall \, t \geq 0, \, \omega \in \Omega,
\end{equation}
is an isomorphism from $ b\mathcal{P}(\Phi',\Psi')$ into $ b\mathcal{P}(\Psi,\Phi)$. 
\end{proposition}

For our proof Proposition \ref{propDualIntegrandsNuclearSpace}  we will need the following result of more general nature: 

\begin{lemma}\label{lemmStrongIntegBarrelled}
If $\Psi$ is a barrelled space, $R:\R_{+} \times \Omega \rightarrow \mathcal{L}(\Psi,\Phi)$ is weakly predictable, and satisfies the condition: $\sup_{(t,\omega)} \abs{ \inner{f}{R(t,\omega) \psi}}< \infty$ $\forall f \in \Phi'$, $\psi \in \Psi$, then  $R \in  b\mathcal{P}(\Psi, \Phi)$.  
\end{lemma}
\begin{proof}
We only need to show that $R$ is weakly bounded. Fix some $f \in \Phi'$. Then, observe that for each $(t,\omega) \in \R_{+} \times \Omega$, the mapping $\psi \mapsto \inner{f}{R(t,\omega) \psi}$ is linear and continuous. Moreover, the condition in $R$ shows that this family is pointwise bounded. Therefore, the Banach-Steinhauss theorem shows that this family is uniformly bounded, that is, the mapping $R$ is weakly bounded.   
\end{proof} 

\begin{proof}[Proof of Proposition \ref{propDualIntegrandsNuclearSpace}]
Since $\Psi$ and $\Phi$ are reflexive spaces, it is clear that for all $t \geq 0$,  $\omega \in \Omega$, $R(t,\omega) \in \mathcal{L}(\Phi',\Psi')$  if and only if $R(t,\omega)' \in \mathcal{L}(\Psi, \Phi)$. Moreover, for all $f \in \Phi'$, $\psi \in \Psi$, $t \geq 0$, $\omega \in \Omega$, we have 
\begin{equation} \label{eqDualityNucleIntegrands}
\inner{f}{R(t,\omega)' \psi} = \inner{R(t,\omega)f}{\psi},
\end{equation}
and hence, by the reflexivity of $\Psi$ and $\Phi$ the weak predictability of $R$ implies that of $R'$ and conversely. In a similar way,  \eqref{eqDualityNucleIntegrands} implies that 
$$ \sup_{(t,\omega)} \abs{ \inner{f}{R(t,\omega)' \psi} } = \sup_{(t,\omega)} \abs{\inner{R(t,\omega)f}{\psi}}, \quad \forall \, f \in \Phi', \, \psi \in \Psi. $$ 
Then, because $\Psi$ and $\Phi$ are barrelled, it follows from Lemma \ref{lemmStrongIntegBarrelled} that $R$ is weakly bounded if and only if $R'$ is weakly bounded. Therefore, the mapping defined in \eqref{eqDefiIsomoNucleIntegrands}
is an isomorphism from $b\mathcal{P}(\Phi',\Psi')$ into $b\mathcal{P}(\Psi,\Phi)$. 
\end{proof}

As a second step in our construction of the stochastic integral we will  show that the spaces $b\mathcal{P}(\Psi, \Phi)$ and $\mathcal{L}(\Psi, \mathcal{L}_{b}(\Phi',b\mathcal{P}))$ are isomorphic (here $\mathcal{L}_{b}(\Phi',b\mathcal{P})$ denotes the space $\mathcal{L}(\Phi',b\mathcal{P})$ equipped with the topology of bounded convergence). In the next result we show that this identification holds if $\Psi$ is bornological (no need to be reflexive nor nuclear) and $\Phi$ is reflexive (no need to be nuclear).

\begin{proposition} \label{propIsomorStrongIntBounPredOpeLinOper}
Let $\Phi$ and $\Psi$ be locally convex spaces. Suppose that $\Psi$ is bornological, and that $\Phi$ is reflexive. Then, the map from $b\mathcal{P}(\Psi,\Phi)$ into $\mathcal{L}(\Psi, \mathcal{L}_{b}(\Phi', b\mathcal{P}))$ defined by 
\begin{equation} \label{defIsomorStrongIntBounPredOpeLinOper}
R \mapsto [ \psi \mapsto \left( f \mapsto \inner{f}{R \psi} \right) ].  
\end{equation}
is a linear isomorphism.  
\end{proposition}
\begin{proof}
Let $R \in b\mathcal{P}(\Psi,\Phi)$. For any given $\psi \in \Psi$, it is clear that $R \psi$ is an element of $b\mathcal{P} (\Phi)$. But then, in view of Proposition 4.7 in \cite{FonsecaMora:StochInteg} it follows that $f \mapsto \inner{f}{R \psi}$ defines an element of $\mathcal{L}_{b}(\Phi', b\mathcal{P})$. It is also clear that the mapping $\psi \mapsto \left( f \mapsto \inner{f}{R \psi} \right)$ is linear. Therefore, in orden to show that the mapping given by \eqref{defIsomorStrongIntBounPredOpeLinOper} is well-defined, we must show that the mapping $\psi \mapsto \left( f \mapsto \inner{f}{R \psi} \right)$ is continuous from $\Psi$  into $\mathcal{L}_{b}(\Phi', b\mathcal{P})$. 

To do this, let $B \subseteq \Psi$ bounded. Because $\sup_{\psi \in B} \sup_{(t,\omega)} \abs{ \inner{f}{R(t,\omega) \psi}}< \infty$  $\forall f  \in \Phi'$ and $\Phi'$ is barrelled (it is reflexive), then by the Banach-Steinhauss theorem (see Theorem 11.9.1 in \cite{NariciBeckenstein}, p.400) the family $\{ f \mapsto \inner{f}{R \psi}: \psi \in B\} \subseteq \mathcal{L}(\Phi', b\mathcal{P})$ is equicontinuous; hence bounded in $\mathcal{L}_{b}(\Phi', b\mathcal{P})$ (see Theorem III.4.1 in \cite{Schaefer}, p.83). Therefore, the operator $\psi \mapsto \left( f \mapsto \inner{f}{R \psi} \right)$ is  bounded (maps bounded subsets of $\Psi$ into bounded subsets in $\mathcal{L}_{b}(\Phi', b\mathcal{P})$), but as $\Psi$ is bornological this implies that this operator is also continuous (see Theorem 13.2.7 in \cite{NariciBeckenstein}, p.444-5). 

Now, it is clear that the mapping defined in  \eqref{defIsomorStrongIntBounPredOpeLinOper} is linear and that has kernel $\{0\}$, so it is injective. It only remain to show that it is surjective. 

Let $S \in \mathcal{L}(\Psi, \mathcal{L}_{b}(\Phi', b\mathcal{P}))$. 
In view of Proposition 4.7 in \cite{FonsecaMora:StochInteg}, to $S$ there corresponds a map $\tilde{S} \in \mathcal{L}(\Psi, b\mathcal{P}(\Phi))$ such that 
\begin{equation} \label{eqEquiOperSPreImagIsoStrongInteg}
S \psi(f)(t,\omega)=\inner{f}{\tilde{S}\psi(t,\omega)}, \quad \forall \, \psi \in \Psi, \, f \in \Phi', \, t \geq 0, \, \omega \in \Omega. 
\end{equation}
Define a collection $R=\{ R(t,\omega): t \geq 0, \omega \in \Omega\}$ of mappings from $\Psi$ into $\Phi$ by means of the prescription
\begin{equation} \label{defiPreImagIsoStrongIntBounLin}
R(t,\omega) \psi \defeq \tilde{S}\psi(t, \omega), \quad \forall \psi \in \Psi, \, t \geq 0, \omega \in \Omega.  
\end{equation}
For any given $t \geq 0, \omega \in \Omega$, the fact that $\tilde{S} \in \mathcal{L}(\Psi, b\mathcal{P}(\Phi))$ implies that $R(t,\omega) \in \mathcal{L}(\Psi,\Phi)$. Moreover, as for each $\psi \in \Psi$, $\tilde{S} \psi \in  b\mathcal{P}(\Phi)$, then $R$ is weakly predictable. Furthermore, because $\tilde{S}$ is a continuous operator it is therefore a bounded operator. That is, given any $B \subseteq \Psi$ bounded, we have that $\{ \tilde{S} \psi: \psi \in B\}$ is bounded in $b\mathcal{P}(\Phi)$. Hence, for every $f \in \Phi'$ we have that 
$$ \sup_{\psi \in B} \sup_{(t,\omega)} \abs{ \inner{f}{R(t,\omega) \psi}}
= \sup_{\psi \in B} \sup_{(t,\omega)} \abs{\inner{f}{\tilde{S}\psi(t,\omega)}} < \infty. $$
Then, $R \in b\mathcal{P}(\Psi,\Phi)$. Finally, it follows from \eqref{eqEquiOperSPreImagIsoStrongInteg} and  \eqref{defiPreImagIsoStrongIntBounLin} that 
$$ S\psi(f)=\inner{f}{\tilde{S}\psi}=\inner{f}{R\psi}, \quad \forall f \in \Phi'.$$
Thus, $R$ is the preimage of $S$ under the mapping \eqref{defIsomorStrongIntBounPredOpeLinOper}. 
\end{proof}

From the previous results we obtain the following conclusion which will be of importance in our construction of the stochastic integral. 

\begin{corollary}\label{coroEquivStrongIntegrands} Let $\Psi$ be a quasi-complete bornological nuclear space and  $\Phi$ a complete barrelled nuclear space. Then the spaces 
 $\mathcal{L}(\Psi, b\mathcal{P}(\Phi))$,  $\mathcal{L}(\Psi, \mathcal{L}_{b}(\Phi', b\mathcal{P}))$, $b\mathcal{P}(\Psi,\Phi)$ and $b\mathcal{P}(\Phi',\Psi')$ are isomorphic as linear spaces.  
 In particular, the map $R \in b\mathcal{P}(\Psi,\Phi) \rightarrow \widetilde{R} 
 \in \mathcal{L}(\Psi, b\mathcal{P}(\Phi))$ is a linear isomorphism, $\widetilde{R}$ denoting the element of $\mathcal{L}(\Psi, b\mathcal{P}(\Phi))$ defined by $\widetilde{R}:  \psi \mapsto R' \psi =(R'(t,\omega) \psi: t \geq 0, \omega \in \Omega)  \in  b\mathcal{P}(\Phi)$.   
\end{corollary}
\begin{proof}
First, recall that a quasi-complete and bornological locally convex space is ultrabornological (see Theorem 13.2.12 in \cite{NariciBeckenstein}, p.449), hence is barrelled as it is the inductive limit of Banach spaces (see Theorem 11.12.2 in \cite{NariciBeckenstein}, p.409). Since, any quasi-complete nuclear space is semireflexive, then $\Psi$ and $\Phi$ are reflexive (see Theorem IV.5.6 in \cite{Schaefer}, p.145). The conclusion now follows combining the results from Propositions  \ref{propDualIntegrandsNuclearSpace} and \ref{propIsomorStrongIntBounPredOpeLinOper}, 
together with Theorem 4.8 in \cite{FonsecaMora:StochInteg} which shows that the mapping $H \in  b\mathcal{P}(\Phi) \rightarrow [f \rightarrow \inner{f}{H}] \in \mathcal{L}_{b}(\Phi', b\mathcal{P})$ is a linear isomorphism.
\end{proof}

We are ready to construct the vector-valued $S^{0}$-stochastic integral with respect to a $S^{0}$-good integrator $X$ in $\Phi'$. A key step will be the following remark: observe that if $R \in b\mathcal{P}(\Phi',\Psi')$, then for each $\psi \in \Psi$ the real-valued stochastic integral $\int \, R' \psi \, dX$ exists since by Corollary \ref{coroEquivStrongIntegrands} we have that $R' \psi \in b\mathcal{P}(\Phi)$.

\begin{theorem}\label{theoConstStrongStochInteg}
Let $\Psi$ be a quasi-complete bornological nuclear space and  $\Phi$ a complete barrelled nuclear spaces. Let  $X=(X_{t}: t \geq 0)$ be a $S^{0}$-good integrator in $\Phi'$. Then for each $R \in b\mathcal{P}(\Phi',\Psi')$ there exists a unique (up to indistinguishable versions) $\Psi'$-valued regular c\`{a}dl\`{a}g $(\mathcal{F}_{t})$-semimartingale $\int \, R \, dX= \left( \int_{0}^{t} R(r) \, dX_{r}: t \geq 0 \right)$, such that for each $\psi \in \Psi$, $\Prob$-a.e. 
\begin{equation}\label{eqConstrStrongStocInteg}
\inner{\int_{0}^{t} R(r) dX_{r}}{\psi}= \int_{0}^{t} \, R'(r) \psi \, dX_{r}, \quad \forall \, t \geq 0. 
\end{equation}
Moreover, if $X(\phi)$ has continuous paths for every $\phi \in \Phi$, then $\int \, R \, dX$ has continuous paths as a  $\Psi'$-valued process. 
\end{theorem}
\begin{proof}
Given $R \in b\mathcal{P}(\Phi',\Psi')$, we have by Corollary \ref{coroEquivStrongIntegrands} that $R$ can be identified with an unique element $\widetilde{R}$ in $\mathcal{L}(\Psi, b\mathcal{P}(\Phi))$ by means of the prescription  $\psi \mapsto \widetilde{R}(\psi) \defeq R' \psi$. 

Now, as $X$ is a $S^{0}$-good integrator, the  stochastic integral mapping $I: b\mathcal{P}(\Phi) \rightarrow S^{0}$, $I(H)= \int \, H \, dX$, is linear continuous. Hence, the mapping $Z=I \circ \widetilde{R}$ defines a cylindrical semimartingale in $\Psi'$ which is linear continuous from $\Psi$ into $S^{0}$. By (\cite{FonsecaMora:Semi}, Theorem 3.7, Proposition 3.14) there exists a unique (up to indistinguishable versions) $\Psi'$-valued regular c\`{a}dl\`{a}g $(\mathcal{F}_{t})$-semimartingale $\int \, R \, dX= \left( \int_{0}^{t} R(r) \, dX_{r}: t \geq 0 \right)$ satisfying that for each $\psi \in \Psi$, $\Prob$-a.e.  
$$ \inner{\int_{0}^{t} R(r) dX_{r}}{\psi}
= Z_{t}(\psi) = (I \circ \widetilde{R}(\psi))_{t} = \int_{0}^{t} \, R'(r) \psi \, dX_{r}, 
$$
$\forall t \geq 0$. This shows \eqref{eqConstrStrongStocInteg}. 

To prove the final assertion assume that $X(\phi)\in S^{c}$ for every $\phi \in \Phi$. Then by Proposition 3.6(1) and Theorem 4.10 in \cite{FonsecaMora:StochInteg} we have that $Z(\psi) \in S^{c}$ for every $\psi \in \Psi$. In such a case by (\cite{FonsecaMora:Semi}, Theorem 3.7, Proposition 3.14) we have that $\int \, R \, dX$ can be chosen such that it satisfies the properties described above and furthermore has continuous paths in $\Psi'$.
\end{proof}

\begin{definition}
We refer to the elements in $b\mathcal{P}(\Phi',\Psi')$ as the  \emph{(operator-valued) stochastic integrands} and for each $R \in b\mathcal{P}(\Phi',\Psi')$ the $\Psi'$-valued process $\int \, R \, dX$ introduced in Theorem \ref{theoConstStrongStochInteg} is called the \emph{vector-valued $S^{0}$-stochastic integral} of $R$. The mapping $R \mapsto \int \, R \, dX$ from $b\mathcal{P}(\Phi',\Psi')$ into $S^{0}(\Psi')$ is referred as the  \emph{(vector-valued) stochastic integral mapping} determined by $X$. 
\end{definition}

\subsection{Properties of the $S^{0}$-stochastic integral}\label{subSectProperStochInteg}

In this section we study some properties of the  $S^{0}$-stochastic integral. Most of the properties are inherited via \eqref{eqConstrStrongStocInteg} from corresponding properties of the real-valued stochastic integral.

\begin{proposition} \label{propImageStrongIntegLineContMapNuclear}
Suppose that $\Upsilon$ and $\Psi$ are quasi-complete bornological nuclear spaces and that $\Phi$ is a complete barrelled nuclear space. Let  $X=(X_{t}: t \geq 0)$ be a $S^{0}$-good integrator in $\Phi'$. Then for each $R \in b\mathcal{P}(\Phi',\Psi')$ and $A \in \mathcal{L}(\Psi', \Upsilon')$ we have $ A \circ R \in b\mathcal{P}(\Phi', \Upsilon')$ and  
\begin{equation}\label{eqImageStroInteLineContMapNuclear}
\int \, A \circ R  \, dX = A \left( \int \, R \, dX \right).
\end{equation}
\end{proposition}
\begin{proof}
 First, since for any $\upsilon \in \Upsilon$, $f \in \Phi'$, $t \geq 0$, $\omega \in \Omega$, we have $ \inner{f}{(A \circ R(t,\omega))'\upsilon}= \inner{f}{R(t,\omega)' A' \upsilon}$, then one can easily check that  $ A \circ R \in b\mathcal{P}(\Phi', \Upsilon')$, hence $\int \, A \circ R \, dX$ exists as an element in $S^{0}(\Upsilon')$ by Theorem \ref{theoConstStrongStochInteg}. 
 
Moreover,  $\forall \upsilon \in \Upsilon$ we have by  \eqref{eqConstrStrongStocInteg} that (as elements in $S^{0}$):
$$ \inner{ \int \, A \circ R \, dX}{\upsilon} = \int \, (R' (A' \upsilon) \, dX = \inner{ \int \, R \, dX}{A' \upsilon} = \inner{A \left( \int \, R \, dX \right)}{\upsilon}. $$
Then  $\int \, A \circ R \, dX$ and $A \left( \int \, R \, dX \right)$ indistinguishable by Remark \ref{remarkEqualInSemimartingale}. 
\end{proof}

The next two results show that the integral is linear both on the integrands as well on the integrators.  

\begin{proposition}\label{propLinerStrongIntegralIntegrands}
Let $\Psi$, $\Phi$ and $X$ as in Theorem \ref{theoConstStrongStochInteg}. Let $R, S \in b\mathcal{P}(\Phi',\Psi')$ and $c \in \R$. Then  the $\Psi'$-valued processes $ \int \, cR+S dX$ and $c\int \, R d X+\int \, S d X$ are indistinguishable. 
\end{proposition}
\begin{proof} First, since  $cR+S \in b\mathcal{P}(\Phi',\Psi')$ the stochastic integrals $ \int \, cR+S dX$, $\int \, R d X$, and $\int \, S d X$ all exists as elements in $S^{0}(\Psi')$. 

Now from the linearity of the real-valued stochastic integral mapping and from \eqref{eqConstrStrongStocInteg}, for each $\psi \in \Psi$ we have (as elements in $S^{0}$):
\begin{eqnarray*}
 \inner{\int \, (cR+S) dX}{\psi} 
& = & \int \, (cR'+S') \psi \, dX \\
& = & c\int \, R' \psi d X+\int \, S' \psi \, d X \\ 
& = & \inner{c\int \, R \, d X+\int \, S \, d X}{\psi}.
\end{eqnarray*}
Then $ \int \, (cR+S) dX$ and $c\int \, R d X+\int \, S d X$ are  indistinguishable by Remark \ref{remarkEqualInSemimartingale}.
\end{proof}
 
\begin{proposition} \label{propLinerStrongIntegralIntegrators}  Let $\Psi$ and $\Phi$ as in Theorem \ref{theoConstStrongStochInteg}. Suppose that $X$ and $Y$ are two $S^{0}$-good integrators in $\Phi'$. Then for each $R \in b\mathcal{P}(\Phi',\Psi')$ the $\Psi'$-valued processes $\int \, R \, d (X+Y)$ and $\int \, R \, d X +\int \, R \, d Y$ are indistinguishable. 
\end{proposition}
\begin{proof}
Let  $R \in b\mathcal{P}(\Phi',\Psi')$. Since $X+Y \in \mathbbm{S}^{0}(\Psi')$  the stochastic integral $\int \, R \, d (X+Y)$ exists as an element in $S^{0}(\Psi')$  by Theorem \ref{theoConstStrongStochInteg}. Finally, that $\int \, R \, d (X+Y)$ and $\int \, R \, d X +\int \, R \, d Y$ are indistinguishable follows similarly as in the proof of  Proposition \ref{propLinerStrongIntegralIntegrands} by using the linearity of the real-valued stochastic integral mapping on the integrators and \eqref{eqConstrStrongStocInteg}. 
\end{proof}

The next result shows that the $S^{0}$-integral behaves well when it comes to take the continuous part and when the process is stopped. 

\begin{proposition}\label{propStrongIntegContPartStopping} Let $\Psi$, $\Phi$ and $X$ as in Theorem \ref{theoConstStrongStochInteg}, and let $\tau$ a stopping time. Then,  for every $R \in b\mathcal{P}(\Phi',\Psi')$  we have
\begin{enumerate}
\item $\displaystyle{\left(\int R \, dX \right)^{c}= \int R \, dX^{c}}$. 
\item \label{propStochIntegStoppingTimes} $\displaystyle{\left(\int R \, dX \right)^{\tau}=  \int R \mathbbm{1}_{[0,\tau]}  \, dX  = \int R \, dX^{\tau}}$. 
\end{enumerate} 
\end{proposition}
\begin{proof} 
First, by Proposition \ref{propGoodIntegraContPartStopping} $X^{c}$ and $X^{\tau}$ are $S^{0}$-good integrators and the stochastic integrals 
$\int R \, dX^{c}$ and $\int R \, dX^{\tau}$ are well-defined as elements in $S^{0}(\Psi')$. 

Now let $R \in b\mathcal{P}(\Phi',\Psi')$. Then for each $\psi \in \Psi$ we have by \eqref{eqConstrStrongStocInteg} and  \eqref{eqContinuPartAndStoppingWeakIntegral} that (as elements in $S^{0}$):
\begin{eqnarray*}
 \inner{\left(\int R \, dX \right)^{c}}{\psi} & = & \left( \inner{\int R \, dX }{\psi} \right)^{c} = \left( \int \, R'\psi \, dX  \right)^{c} \\
& = &  \int \, R' \psi \, dX^{c} = \inner{\int \, R \, dX^{c}}{\psi}. 
\end{eqnarray*}
Then $\left(\int R \, dX \right)^{c}$ and $\int R \, dX^{c}$ are indistinguishable by Remark \ref{remarkEqualInSemimartingale}. Similarly, we obtain \ref{propStochIntegStoppingTimes}. 
\end{proof}


In the following result we gather some permanence properties of the stochastic integral for particular classes of integrators. 

\begin{proposition}\label{propGoodIntegratorsNuclearSpace}
Let $\Psi$ be a quasi-complete bornological nuclear space and  $\Phi$ a complete barrelled nuclear space. Let $X=(X_{t}: t \geq 0)$ be a $\Phi'$-valued $(\mathcal{F}_{t})$-adapted semimartingale such that the mapping $\phi \mapsto X(\phi)$ is continuous from $\Phi$ into $S^{0}$. For each $R \in b\mathcal{P}(\Phi',\Psi')$ we have:  
\begin{enumerate}
\item \label{integralHpSemimartingale} If $X$ is a  $\mathcal{H}^{p}_{S}$-semimartingale, then $\int \, R \, dX$ is a $\mathcal{H}^{p}_{S}$-semimartingale. Furthermore, there exists a continuous Hilbertian seminorm $q$ on such that $\int \, R \, dX$ has a  $\Phi'_{q}$ version which is a c\`{a}dl\`{a}g $\mathcal{H}^{p}_{S}$-semimartingale.
\item \label{integralSquareMartingale} If $X$ is a $\mathcal{M}_{\infty}^{2}$-martingale, then $\int \, R \, dX$ is a  $\mathcal{M}_{\infty}^{2}$-martingale. Furthermore, there exists a continuous Hilbertian seminorm $q$ on such that $\int \, R \, dX$ has a  $\Phi'_{q}$ version which is a c\`{a}dl\`{a}g $\mathcal{M}_{\infty}^{2}$-martingale.
\item \label{integralIntegrableVariation} If $X$ is a $\mathcal{A}$-semimartingale, then $\int \, R \, dX$ is a $\mathcal{A}$-semimartingale. Furthermore, there exists a continuous Hilbertian seminorm $q$ on such that $\int \, R \, dX$ has a $\Phi'_{q}$ version which is a c\`{a}dl\`{a}g $\mathcal{A}$-semimartingale. 
\end{enumerate}
\end{proposition}
\begin{proof}
To prove \ref{integralHpSemimartingale}, by \eqref{eqConstrStrongStocInteg} and Proposition 4.12(1) in \cite{FonsecaMora:Semi} we have that $X$ is a $S^{0}$-good integrator and $\int \, R \, dX$ is a $\mathcal{H}^{p}_{S}$-semimartingale. The existence of the continuous Hilbertian  seminorm $q$ and the $\Phi'_{q}$-valued c\`{a}dl\`{a}g version which is a $\mathcal{H}^{p}_{S}$-semimartingale follows by Proposition 3.22 in \cite{FonsecaMora:Semi}. 

Similarly, \ref{integralSquareMartingale} follows from 
 Proposition 4.12(2) in \cite{FonsecaMora:Semi} and Theorem 5.2 in \cite{FonsecaMora:Existence}. Finally, \ref{integralIntegrableVariation} follows from  Proposition 4.12(3) in \cite{FonsecaMora:Semi} and Proposition 3.22 in \cite{FonsecaMora:Semi}. 
\end{proof}

\subsection{Extension of the $S^{0}$-stochastic integral}\label{sectExtStochIntegral}

In this section we carry out a extension of the $S^{0}$-stochastic integral for integrands that are locally bounded in the following sense:

\begin{definition}\label{defiLocallyBoundedIntegrands}
Given two locally convex spaces $\Psi$ and $\Phi$, denote by $\mathcal{P}_{loc}(\Phi',\Psi')$ the collection of all the mappings $R: \R_{+} \times \Omega \rightarrow \mathcal{L}(\Phi',\Psi')$ for which there exists a sequence $(\tau_{n}: n \in \N)$ of stopping times increasing to $\infty$ $\Prob$-a.e. such that for each $n \in \N$, $(t,\omega) \mapsto R^{\tau_{n}}(t,\omega) \defeq R(t \wedge \tau_{n}(\omega), \omega)$ belongs to $b \mathcal{P}(\Phi', \Psi')$. We call $(\tau_{n}: n \in \N)$ a \emph{localizing sequence} for $R$. 
\end{definition}

It should be clear that $b \mathcal{P}(\Phi', \Psi') \subseteq  \mathcal{P}_{loc}(\Phi', \Psi')$. Moreover, observe that if $R:\R_{+} \times \Omega \rightarrow \mathcal{L}(\Phi',\Psi')$ is such that $R \in \mathcal{P}_{loc}(\Phi', \Psi')$ we have $ \inner{R f}{ \psi} \in \mathcal{P}_{loc}$ for every $f \in \Phi'$ and $\psi \in \Psi$. The converse is less clear to hold because   the localizing sequence for  $\inner{R f}{ \psi} $ might depend on $f$ and $\psi$.  

The following result will be useful to generate examples of locally bounded integrands. 
Let $H: \R_{+} \times \Omega \rightarrow \Phi$ and $G: \R_{+} \times \Omega \rightarrow \Psi'$. Define $R: \R_{+} \times \Omega \rightarrow \mathcal{L}(\Phi', \Psi')$ by $R(t,\omega)=H(t,\omega) \otimes G(t,\omega)$, that is
\begin{equation}\label{eqDefiIntegrandAsTensorProduct}
R(t,\omega)f=\inner{f}{H(t,\omega)} G(t,\omega), \quad \forall t \geq 0, \, \omega \in \Omega, \, f \in \Phi'. 
\end{equation}

\begin{proposition}\label{propConstrucPredicLocallyBoundedByTensors} Let $\Psi$ and $\Phi$ be two reflexive locally convex spaces. For $R$ as in \eqref{eqDefiIntegrandAsTensorProduct} we have:
\begin{enumerate}
\item If $H \in b\mathcal{P}(\Phi)$ and $G \in b\mathcal{P}(\Psi')$, then $R \in b \mathcal{P}(\Phi', \Psi')$.
\item If $H \in \mathcal{P}_{loc}(\Phi)$ and $G \in b\mathcal{P}_{loc}(\Psi')$, then $R \in  \mathcal{P}_{loc}(\Phi', \Psi')$.
\end{enumerate} 
\end{proposition}
\begin{proof}
(1) Assume $H \in b\mathcal{P}(\Phi)$ and $G \in b\mathcal{P}(\Psi')$. We must show $R$ is weakly predictable and weakly bounded. To see why $R$ is weakly predictable, let $f \in \Phi'$ and $\psi \in \Psi$. Then by the weak predictability of $H$ and $G$ it follows that the mapping 
$$(t,\omega) \mapsto \inner{R(t,\omega)f}{\psi} = \inner{f}{H(t,\omega)}\inner{G(t,\omega)}{\psi}, $$
is predictable. 

Now, to show that $R$ is weakly bounded, let $f \in \Phi'$ and $\psi \in \Psi$ Then, since $H$ and $G$ are weakly bounded we have
$$
\sup_{(t,\omega)} \abs{\inner{R(t,\omega)f}{\psi}} 
= \sup_{(t,\omega)} \abs{\inner{f}{H(t,\omega)}} \cdot \sup_{(t,\omega)} \abs{\inner{G(t,\omega)}{\psi}} < \infty.  $$
Finally,  $R$ is weakly bounded by Lemma \ref{lemmStrongIntegBarrelled}. 

(2) Let $(\sigma_{n}: n \in \N)$ and $(\nu_{n}: n \in \N)$ be localizing sequences for $H$ and $G$ respectively. Then, $\tau_{n}=\sigma_{n} \wedge \nu_{n}$ for $n\in \N$, is a localizing sequence for both $H$ and $G$. Then by (1) we have $R^{\tau_{n}} \in  b \mathcal{P}(\Phi', \Psi')$ for every $n \in \N$. Hence $R \in \mathcal{P}_{loc}(\Phi', \Psi')$.
\end{proof}

\begin{definition}
For a locally convex space $\Phi$, we denote by $\mathbbm{L}(\Phi)$ (respectively by $\mathbbm{D}(\Phi)$) the collection of all $\Phi$-valued $(\mathcal{F}_{t})$-adapted processes with c\`{a}gl\`{a}d paths (respectively with c\`{a}dl\`{a}g paths).
\end{definition}

For $Y=(Y_{t}: t \geq 0)$ a $\Phi$-valued process whose paths possesses left-limits, let $Y_{t-} \defeq \lim_{u \rightarrow s, u<s} Y_{u}$ and $Y_{0-} \defeq 0$. Observe that if $Y=(Y_{t}: t \geq 0) \in \mathbbm{D}(\Phi)$, then $Y_{-}=(Y_{t-}: t \geq 0) \in \mathbbm{L}(\Phi)$.

\begin{example}
Let $\varphi \in  \mathscr{D}(\R^{d})$ and let $Z=(Z_{t}: t \geq 0)$ be a $\R^{d}$-valued semimartingale. Define a $\mathscr{D}(\R^{d})$-valued process $H=(H(t,\omega): t\geq 0, \omega \in \Omega)$  by $H(t,\omega)=\varphi(\cdot + Z_{t-}(\omega))$. As shown in Example 5.7 in \cite{FonsecaMora:StochInteg}, we have $H \in \mathcal{P}_{loc}(\mathscr{D}(\R^{d}))$. 

Let $X=(X_{t}: t \geq 0)$ be a $\mathscr{D}'(\R^{d})$-valued c\`{a}dl\`{a}g weakly $(\mathcal{F}_{t})$-adapted process. Let $G=(G(t,\omega): t\geq 0, \omega \in \Omega)$ be the $\mathscr{D}'(\R^{d})$-valued process given by $G(t,\omega)=X_{t-}(\omega)$. By Example 5.5 in \cite{FonsecaMora:StochInteg}, we have $G \in \mathcal{P}_{loc}(\mathscr{D}'(\R^{d}))$. Then by Proposition \ref{propConstrucPredicLocallyBoundedByTensors} we have $R \in  \mathcal{P}_{loc}(\mathscr{D}'(\R^{d}), \mathscr{D}'(\R^{d}))$, for  $R(t,\omega)=H(t,\omega) \otimes G(t,\omega)$, that is
$$ R(t,\omega)f =\inner{f}{\varphi(\cdot + Z_{t-}(\omega))} X_{t-}(\omega), \quad \forall t \geq 0, \, \omega \in \Omega, \, f \in \mathscr{D}'(\R^{d}).  $$
\end{example}

The extension of the vector-valued $S^{0}$-stochastic integral is carried out in the next result.

\begin{theorem}\label{theoStochIntegLocallyBounded}
Let $\Psi$ be a quasi-complete bornological nuclear space and  $\Phi$ a complete barrelled nuclear space. Let  $X=(X_{t}: t \geq 0)$ be a $S^{0}$-good integrator in $\Phi'$. Then for each $R \in \mathcal{P}_{loc}(\Phi',\Psi')$ there exists a unique (up to indistinguishable versions) $\Psi'$-valued regular c\`{a}dl\`{a}g $(\mathcal{F}_{t})$-semimartingale $\int \, R \, dX= \left( \int_{0}^{t} R(r) \, dX_{r}: t \geq 0 \right)$, such that for each $\psi \in \Psi$, $\Prob$-a.e. 
\begin{equation}\label{eqWeakStrongLocalBounded}
\inner{\int_{0}^{t} R(r) dX_{r}}{\psi}= \int_{0}^{t} \, R'(r) \psi \, dX_{r}, \quad \forall \, t \geq 0. 
\end{equation}
Moreover, 
\begin{enumerate}
\item $\displaystyle{\left(\int R \, dX \right)^{c}= \int R \, dX^{c}}$. 
\item $\displaystyle{\left(\int R \, dX \right)^{\tau} =  \int R \mathbbm{1}_{[0,\tau]}  \, dX = \int R \, dX^{\tau}}$, for every stopping time $\tau$. 
\item The mapping $\displaystyle{(R,X) \mapsto \int \, R \, dX}$ is bilinear. 
\end{enumerate}
Furthermore, if  $X(\phi)$ has continuous paths for every $\phi \in \Phi$, then $\int \, R \, dX$ has continuous paths as a  $\Psi'$-valued process. 
\end{theorem}
\begin{proof}
Given $R \in \mathcal{P}_{loc}(\Phi',\Psi')$ and a sequence  $(\tau_{n}: n \in \N)$ of stopping times as in Definition \ref{defiLocallyBoundedIntegrands}, it is clear that $R \mathbbm{1}_{[0,\tau_{n}]} \in b \mathcal{P}(\Phi',\Psi')$ for each $n \in \N$. By Theorem \ref{theoConstStrongStochInteg} each 
$ \int \, R \mathbbm{1}_{[0,\tau_{n}]} \, dX$ is a well-defined element in $S^{0}(\Psi')$ satisfying \eqref{eqConstrStrongStocInteg} for $R \mathbbm{1}_{[0,\tau_{n}]}$.

Then for each $t \geq 0$ we can define 
 $$ \int_{0}^{t} \, R \, dX = \int_{0}^{t} \, R \mathbbm{1}_{[0,\tau_{n}]} \,  dX, $$
for any $n \in \N$ such that $\tau_{n} \geq t$. Using Theorem \ref{propStrongIntegContPartStopping}(2) one can verify by following a standard localization argument (as for example in \cite{DaPratoZabczyk}, Chapter 6) that this definition for $\int \, R \, dX$ is consistent and that it is  independent of the localizing sequence for $R$. Since the property of being an element in $S^{0}(\Psi')$ is stable by localization, then  $\int \, R \, dX \in S^{0}(\Psi')$. In case $X(\phi) \in S^{c}$ for every $\phi \in \Phi$, then $\int \, R \, dX$ has continuous paths as a  $\Psi'$-valued process by Theorem \ref{theoConstStrongStochInteg}. 

Finally that properties  (1)-(3) are satisfied follows from Propositions
\ref{propLinerStrongIntegralIntegrands} and \ref{propLinerStrongIntegralIntegrators}, and Theorem \ref{propStrongIntegContPartStopping} by choosing an appropriate localizing sequence. 
\end{proof}

\begin{remark}
By choosing an appropriate localizing sequence we can check that all the properties of the stochastic integral listed in Section \ref{subSectProperStochInteg} (except Proposition \ref{propGoodIntegratorsNuclearSpace}) have an analogue for integrands in $R \in \mathcal{P}_{loc}(\Phi',\Psi')$. 
\end{remark}

\begin{example}\label{examVectorStochIntegRealValueSemimar}
In this example we apply our theory to construct stochastic integrals for vector-valued integrands with respect to real-valued semimartingales. An application for this construction will be given in Section \ref{sectItoFormula}. 

Let $z=(z_{t}: t \geq 0)$ be a real-valued semimartingale. By the Bichteler-Dellacherie theorem $z$ is a $S^{0}$-good integrator. Let $R \in \mathcal{P}_{loc}(\Psi')$.  Since the scalar multiplication is a continuous operation, we can regard $R(t,\omega)$ as an element in $\mathcal{L}(\R,\Psi')$ via the prescription $a \in \R \mapsto a \cdot R(t,\omega)$. With the above identification it is clear that $R \in \mathcal{P}_{loc}(\R,\Psi')$. 

 Therefore, by Theorem \ref{theoStochIntegLocallyBounded} we can define the stochastic integral $\int_{0}^{t} \, R(r) \, dz_{r}$, $t \geq 0$, which is a $\Psi'$-valued regular  c\`{a}dl\`{a}g $(\mathcal{F}_{t})$-semimartingale. Moreover, by \eqref{eqWeakStrongLocalBounded}, $\Prob$-a.e. for every $\psi \in \Psi$ and $t \geq 0$, we have 
\begin{equation}\label{eqWeakStrongCompaIntegralRealSemima}
\inner{\int_{0}^{t} \, R(r) \, dz_{r}}{\psi}= \int_{0}^{t} \, \inner{\, R(r)}{\psi} \, dz_{r},
\end{equation}
where the integral in the right-hand side of \eqref{eqWeakStrongCompaIntegralRealSemima} is a (classical) real-valued stochastic integral with respect to $z$.
%
\end{example}



\subsection{Riemann Representation}\label{sectRiemannRepresentation}

Our next objective is to prove that the $S^{0}$-stochastic integral satisfies a Riemann representation. We adopt the following terminology from the definition from \cite{Protter} (see \cite{Protter}, II.5, p.64); here we state $(2)$ for $\Phi$-valued processes. 

\begin{definition} \hfill
\begin{enumerate}
\item A \emph{random partition} $\sigma$ is a finite sequence of finite stopping times:
$$ 0=\tau_{0} \leq \tau_{1} \leq \cdots \leq \tau_{m+1} < \infty. $$
\item For a locally convex space $\Phi$, given a $\Phi$-valued process $R$ and a random partition $\sigma$, we define the process $R$ \emph{sampled} at $\sigma$ to be
\begin{equation}\label{eqDefiSampledProcessLCS}
 R^{\sigma} \defeq R_{0} \mathbbm{1}_{\{0\}}+\sum_{k=1}^{m} R_{\tau_{k}} \mathbbm{1}_{(\tau_{k},\tau_{k+1}]}.
\end{equation}
\item A sequence of random partitions $(\sigma_{n})$, 
$$ \sigma_{n}: \tau_{0}^{n} \leq \tau_{1}^{n} \leq \dots \leq \tau^{n}_{m_{n}+1}, $$
is said to \emph{tend to the identity} if 
\begin{enumerate}
\item $\displaystyle{\lim_{n} \sup_{k} \tau^{n}_{k}=\infty}$ a.s., and
\item $\norm{\sigma_{n}}=\sup_{k} \abs{\tau_{k+1}^{n}-\tau_{k}^{n}} \rightarrow 0$ a.s.
\end{enumerate}  
\end{enumerate}
\end{definition}

We recall the notion of UCP convergence for a sequence of stochastic processes taking values in the dual of a locally convex space. 

\begin{definition}
Let $\Psi$ be a Hausdorff locally convex space and let $\Pi$ denote a system of seminorms generating the  topology on $\Psi'$. Let $X$ and $(X^{n}: n \in \N)$, with $X^{n} =(X^{n}_{t}: t \geq 0)$, be $\Psi'$-valued c\`{a}dl\`{a}g  processes. 
We say that $X^{n}$ converges to $X$ \emph{uniformly on compacts in probability}, abbreviated as $X^{n} \overset{ucp}{\rightarrow} X$, if for every choice of $T>0$, $\epsilon >0$, and every continuous seminorm $p$ on $\Psi'$ we have 
$$ \lim_{n \rightarrow \infty} \Prob \left( \sup_{0 \leq t \leq T} p(X^{n}_{t}-X_{t}) \geq \epsilon \right)=0.  $$
\end{definition}

The reader is referred to \cite{FonsecaMora:UCPConvergence} for further properties and sufficient conditions for the UCP convergence for a sequence of $\Psi'$-valued processes in the dual of a nuclear space $\Psi'$.

We are ready to state the main result of this section, which provides an extension of the usual theorem for classical stochastic integrals (see Theorem 21, p.64 of \cite{Protter}) to vector-valued $S^{0}$-stochastic integrals, by using $S^{0}$-good integrators.    

\begin{theorem}[Approximation by Riemann sums]\label{theoRiemannRepresentation}
Suppose that $\Psi$ is a quasi-complete bornological nuclear space and that $\Phi$ is a complete barrelled nuclear space whose strong dual $\Phi'$ is also nuclear. Let  $X=(X_{t}: t \geq 0)$ be a $S^{0}$-good integrator in $\Phi'$. Assume that $R$ is a process in $\mathbbm{L}(\mathcal{L}_{b}(\Phi',\Psi')) \cap \mathcal{P}_{loc}(\Phi',\Psi')$ or in $\mathbbm{D}(\mathcal{L}_{b}(\Phi',\Psi')) \cap \mathcal{P}_{loc}(\Phi',\Psi')$. Let $(\sigma_{n})$ be a sequence of random partitions tending to the identity. Then, 
\begin{equation}\label{eqIntegOfStoppedProcess}
\int_{0}^{t} \, R^{\sigma_{n}} \, dX =R(0)(X_{0})  +\sum_{k=1}^{m_{n}} \, R(\tau^{n}_{k}\wedge t) \left( X_{\tau^{n}_{k+1} \wedge t}-X_{\tau^{n}_{k} \wedge t} \right),   
\end{equation}
and 
\begin{equation}\label{eqUCPConvergRiemannRepre}
\int \, R^{\sigma_{n}} \, dX \overset{ucp}{\rightarrow} \int \, R_{-} \, dX.
\end{equation}
\end{theorem}
\begin{proof}
Assume $R \in \mathbbm{L}(\mathcal{L}_{b}(\Phi',\Psi')) \cap \mathcal{P}_{loc}(\Phi',\Psi')$. We will show first that \eqref{eqIntegOfStoppedProcess} holds true. Given $\psi \in \Psi$, let $R'\psi=(R'\psi(t,\omega): t \geq 0, \omega \in \Omega)$ be defined by $R'\psi(t,\omega) \defeq R(t,\omega)'\psi \in \Phi$. We must prove that $R' \psi  \in \mathbbm{L}(\Phi) \cap \mathcal{P}_{loc}(\Phi)$. 

First, given a stopping time $\tau$ by Corollary \ref{coroEquivStrongIntegrands} we have $(R'\psi)^{\tau} \in b\mathcal{P}(\Phi)$. Hence by choosing a localizing sequence we can easily verify that  $R' \psi  \in \mathcal{P}_{loc}(\Phi)$. 

We must prove $R'\psi \in \mathbbm{L}(\Phi)$. Let $\omega \in \Omega$ such that the mapping $t \mapsto R(t,\omega) \in \mathcal{L}_{b}(\Phi',\Psi')$ is left-continuous. Let $\psi \in \Psi$ and $B \subseteq \Phi'$ bounded. Then, since $p_{\psi}(g)=\abs{\inner{g}{\psi}}$ is a continuous seminorm on $\Phi'$ and $q_{B}(\varphi)=\sup_{f \in B}\abs{\inner{f}{\varphi}}$ is a continuous seminorm on $\Psi$, by the left-continuity of $t \mapsto R(t,\omega) $ we have 
\begin{eqnarray*}
 \lim_{s \nearrow t} q_{B}(R'\psi(s,\omega)-R'\psi(t,\omega))
 & = & \lim_{s \nearrow t} \sup_{f \in B}\abs{\inner{f}{R'\psi(s,\omega)-R'\psi(t,\omega)}} \\
 & = &  \lim_{s \nearrow t} \sup_{f \in B}\abs{\inner{R(s,\omega)f -R(t,\omega) f}{\psi}} \\
 & = &  \lim_{s \nearrow t} \sup_{f \in B} p_{\psi} (R(s,\omega)f -R(t,\omega) f) =  0.
\end{eqnarray*} 
Moreover, since $\Psi$ is reflexive the collection $q_{B}$, for $B \subseteq \Phi'$ bounded, generates the topology on $\Psi$. Therefore we have shown that the mapping $t \mapsto R'\psi(t,\omega) \in \Psi$ is left-continuous. Thus we conclude  $R' \psi  \in \mathbbm{L}(\Phi)$. 

Now since we have $R' \psi  \in \mathbbm{L}(\Phi) \cap \mathcal{P}_{loc}(\Phi)$, by Theorem 5.14 in \cite{FonsecaMora:StochInteg}, and then by the definition of $R'\psi$ we have
\begin{eqnarray*}
\int_{0}^{t} \, (R'\psi)^{\sigma_{n}} \, dX 
& = & \inner{X_{0}}{R'\psi(0)} +\sum_{k=1}^{m_{n}} \, \inner{X_{\tau^{n}_{k+1} \wedge t}-X_{\tau^{n}_{k} \wedge t}}{R'\psi(\tau^{n}_{k}\wedge t)} \\
& = & \inner{R(0)(X_{0})}{\psi} +\sum_{k=1}^{m_{n}} \, \inner{R(\tau^{n}_{k}\wedge t) \left( X_{\tau^{n}_{k+1} \wedge t}-X_{\tau^{n}_{k} \wedge t} \right)}{\psi} \\
& = & \inner{R(0)(X_{0})  +\sum_{k=1}^{m_{n}} \, R(\tau^{n}_{k}\wedge t) \left( X_{\tau^{n}_{k+1} \wedge t}-X_{\tau^{n}_{k} \wedge t} \right)}{\psi}. 
\end{eqnarray*}

Now, observe that by \eqref{eqDefiSampledProcessLCS} first applied to $R$ and $\sigma_{n}$, then to $R'\psi$ and $\sigma_{n}$, we have
$$ (R^{\sigma_{n}})' \psi =   R(0)' \psi \mathbbm{1}_{\{0\}}+\sum_{k=1}^{m} R(\tau_{k})' \psi  \mathbbm{1}_{(\tau_{k},\tau_{k+1}]} = (R'\psi)^{\sigma_{n}}. $$

Then by \eqref{eqWeakStrongLocalBounded} we have 
\begin{equation}\label{eqStoppedWeakStrongCompati}
\inner{ \int_{0}^{t} \, R^{\sigma_{n}} \, dX}{\psi} = \int_{0}^{t} \, (R'\psi)^{\sigma_{n}} \, dX. 
\end{equation}

Hence, for every $\psi \in \Psi$, 
$$ \inner{ \int_{0}^{t} \, R^{\sigma_{n}} \, dX}{\psi} = \inner{R(0)(X_{0})  +\sum_{k=1}^{m_{n}} \, R(\tau^{n}_{k}\wedge t) \left( X_{\tau^{n}_{k+1} \wedge t}-X_{\tau^{n}_{k} \wedge t} \right)}{\psi}. $$

Therefore the left and right hand sides in 
\eqref{eqIntegOfStoppedProcess} are indistinguishable processes by Remark \ref{remarkEqualInSemimartingale}. This proves \eqref{eqIntegOfStoppedProcess}. 

Now we prove \eqref{eqUCPConvergRiemannRepre}. Given $\psi \in \Psi$, by \eqref{eqStoppedWeakStrongCompati}, the Riemann representation formula for the real-valued stochastic integral (Theorem 5.14 in \cite{FonsecaMora:StochInteg}) and  \eqref{eqWeakStrongLocalBounded},
we have that
$$ \inner{ \int_{0}^{t} \, R^{\sigma_{n}} \, dX}{\psi} = \int_{0}^{t} \, (R'\psi)^{\sigma_{n}} \, dX \overset{ucp}{\rightarrow}
\int \, R' \psi \, dX = \inner{\int \, R \, dX}{\psi}
$$
Therefore by Proposition 3.7 in \cite{FonsecaMora:UCPConvergence}  we conclude that \eqref{eqUCPConvergRiemannRepre} holds.

The case $R \in \mathbbm{D}(\mathcal{L}_{b}(\Phi',\Psi')) \cap \mathcal{P}_{loc}(\Phi',\Psi')$ can be obtained from the former case by replacing $R$ with $R_{-} \in \mathbbm{L}(\mathcal{L}_{b}(\Phi',\Psi')) \cap \mathcal{P}_{loc}(\Phi',\Psi')$. 
\end{proof}

\subsection{The integral as a $S^{0}$-good integrator}\label{sectIntegAsAGoodIntegrator}

In this section we explore sufficient conditions for the stochastic integral to be a $S^{0}$-good integrator. 
We start with the following result which shows that for some particular classes of semimartingales the stochastic integral is always a $S^{0}$-good integrator. 

\begin{proposition}\label{propParticuSemimartIntegralGoodItegrator}
Let $\Psi$ be a complete bornological barrelled nuclear space and let $\Phi$ be a complete barrelled nuclear space. Assume that the $\Phi'$-valued process $X=(X_{t}: t \geq 0)$  satisfies any of the following conditions:
\begin{enumerate}
\item $X$ is locally a $\mathcal{H}^{p}_{S}$-semimartingale.  
\item $X$ is locally a $\mathcal{M}_{\infty}^{2}$-martingale. 
\item $X$ is locally a $\mathcal{A}$-semimartingale.
\end{enumerate}
Then, for every $R \in  \mathcal{P}_{loc}(\Phi',\Psi')$ the stochastic integral $Y=\int \, R \, dX$ is a $S^{0}$-good integrator in $\Psi'$.
\end{proposition}
\begin{proof}
By Proposition \ref{propGoodIntegratorsNuclearSpace} and Theorem \ref{theoStochIntegLocallyBounded}(2) if $X$ is locally a $\mathcal{H}^{p}_{S}$-semimartingale (respectively locally a $\mathcal{M}_{\infty}^{2}$-martingale or locally a $\mathcal{A}$-semimartingale) then the stochastic integral $Y$ is  locally a $\mathcal{H}^{p}_{S}$-semimartingale (respectively locally  a $\mathcal{M}_{\infty}^{2}$-martingale or locally a $\mathcal{A}$-semimartingale). Hence, by Corollary \ref{coroConcreteLocalGoodIntegrators} we have $Y$ is a $S^{0}$-good integrator. 
\end{proof}

Let $\Psi$ and $\Phi$ as in Proposition \ref{propParticuSemimartIntegralGoodItegrator}. 
In the case that we assume only that $X$ is a $S^{0}$-good integrator, it is not clear that the stochastic integral $Y=\int \, R \, dX$ is a $S^{0}$-good integrator in $\Psi'$ for each $R \in  \mathcal{P}_{loc}(\Phi',\Psi')$. We will show however that if we additionally assume that $\Psi$ has the bounded approximation property then the stochastic integral is always a $S^{0}$-good integrator (see Theorem \ref{theoIntegralGoodIntegrator}). 

Recall that $\Psi$ has the \emph{bounded approximation property} if there exists an equicontinuous net $(A_{\gamma}: \gamma \in \Gamma) \subseteq \mathcal{L}(\Psi, \Psi)$ with $\mbox{dim}(A_{\gamma}(\Psi))< \infty$ for every $\gamma \in \Gamma$ and $\lim_{\gamma \in \Gamma} A_{\gamma}(\psi)=\psi$ for every $\psi \in \Psi$. In other words, the net $(A_{\gamma}: \gamma \in \Gamma)$ converges to the identity in the topology of simple convergence. Being $\Psi$ a Montel space and as a consequence of  the Banach-Steinhaus theorem (e.g. Theorem 11.9.4 in \cite{NariciBeckenstein}), the convergence also occurs in the topology of bounded convergence. Therefore, for any given $R \in \mathcal{L}(\Psi,b\mathcal{P}(\Phi))$ the equicontinuous sequence $(R\circ A_{\gamma}: \gamma \in \Gamma) \subseteq  \mathcal{L}(\Psi,b\mathcal{P}(\Phi))$ of finite rank operators converges to $R$ in the topology of bounded convergence. Since $\mathcal{L}(\Psi,b\mathcal{P}(\Phi))\simeq b\mathcal{P}(\Phi',\Psi')$ (Corollary \ref{coroEquivStrongIntegrands}), we obtain the following result:

\begin{lemma}\label{lemmaEquiNetSimpleBoundApproxPrope}
Let $\Psi$ be a quasi-complete bornological nuclear space with the bounded approximation property and let $\Phi$ be a complete barrelled nuclear space. For any given $R \in b\mathcal{P}(\Phi',\Psi')$ there exists an equicontinuous sequence of elementary integrands $(R_{\gamma}: \gamma \in \Gamma)$ converging to $R$ in the topology of bounded convergence.  
\end{lemma}

\begin{remark}\label{remaBoundedApproxima}
It is well-known that because $\Psi$ is barrelled, a sufficient condition for it to possess the bounded approximation property is the existence of a Schauder basis. In effect, let $(\varphi_{n}: n \in \N) \subseteq \Psi$ be a Schauder basis with coefficient functionals $(g_{n}: n \in \N) \subseteq \Psi'$; that is $\inner{g_{m}}{\varphi_{n}}=\delta_{m,n}$ and $\psi=\sum_{n=1}^{\infty} g_{n}(\psi)\varphi_{n}$ converges in $\Psi$ for each $\psi \in \Psi$. Denote by $P_{n}:\Psi \rightarrow \Psi$ the mapping $P_{n}(\psi)\defeq   \sum_{k=1}^{n} g_{k}(\psi)\varphi_{k}$, which is a continuous projection onto $\mbox{span}\{\varphi_{1}, \dots, \varphi_{n}\}$. By definition, $P_{n}(\psi) \rightarrow \psi$ as $n \rightarrow \infty$ and because $\Psi$ is barrelled, by the Banach-Steinhauss theorem the family $(P_{n}: n \in \N)$ is equicontinuous. Thus, $\Psi$ has the bounded approximation property. For example,   $\mathcal{C}^{\infty}([-1,1])$ possesses an Schauder basis (see \cite{Jarchow}, Section 14.8), hence have the bounded approximation property.
\end{remark}

We return to our study of the  $S^{0}$-good integrator property for the stochastic integral. 
Observe that under the assumptions in Theorem \ref{theoStochIntegLocallyBounded} the stochastic integral $Y= \int \, R \, dX $ is a  $\Psi'$-valued regular c\`{a}dl\`{a}g $(\mathcal{F}_{t})$-semimartingale. Since $\Psi$ is assumed to be ultrabornological (hence is barrelled), by Theorem 2.10 in \cite{FonsecaMora:Existence} and  Proposition 3.15 in \cite{FonsecaMora:Semi} the mapping 
 $\psi \mapsto Y(\psi)=\inner{\int \, R \, dX}{\psi} $ is continuous from $\Psi$ into $S^{0}$. Then, in order to show $Y$ is a $S^{0}$-good integrator it only remains to show that the mapping $H \mapsto \int \, H \, dY$ is continuous from $b\mathcal{P}(\Psi)$ into $S^{0}$. 
  
\begin{theorem}\label{theoIntegralGoodIntegrator}
Let $\Psi$ be a quasi-complete bornological nuclear space with the bounded approximation property and let $\Phi$ be a complete barrelled nuclear space. Let  $X=(X_{t}: t \geq 0)$ be a $S^{0}$-good integrator in $\Phi'$ and let $R \in  \mathcal{P}_{loc}(\Phi',\Psi')$. 
Then $Y=\int \, R \, dX$ is a $S^{0}$-good integrator in $\Psi'$. Moreover, for every $H \in \mathcal{P}_{loc}(\Psi)$ we have 
\begin{equation}\label{eqAssociaStochIntegral}
 \int \, H \, dY = \int \, R'(H) \, dX,
\end{equation}
where $R'(H)(t,\omega)=R'(t,\omega)H(t,\omega)$ for all $t \geq 0$, $\omega \in \Omega$.   
\end{theorem}
\begin{proof}
First, since the property of being a $S^{0}$-good integrator is preserved under localization (Corollary \ref{coroLocalGoodIntegrator}) and by the compatibility of the stochastic integral with stopping times (Theorem \ref{theoStochIntegLocallyBounded}), it suffices to show that the theorem holds true  under the assumption that $R \in  b\mathcal{P}(\Phi',\Psi')$ and $H \in b\mathcal{P}(\Psi)$. 

We start by showing that \eqref{eqAssociaStochIntegral} holds true for every elementary $H \in b\mathcal{P}(\Psi)$ and any $R \in  b\mathcal{P}(\Phi',\Psi')$. In effect, let $H$ be of the form  
$$H(t,\omega)=\sum_{k=1}^{m} h_{k}(t,\omega) \psi_{k},$$
where for $k=1, \cdots, m$ we have $h_{k} \in b\mathcal{P}$ and $\psi_{k} \in \Psi$. 
By \eqref{eqActionWeakIntegSimpleIntegNuclear}, \eqref{eqConstrStrongStocInteg} and Proposition 4.16 in \cite{FonsecaMora:StochInteg} we have
\begin{eqnarray*}
  \int \, H \, dY & = &   \sum_{k=1}^{n} \, h_{k} \cdot \inner{Y}{\phi_{k}} = \sum_{k=1}^{n} \, h_{k} \cdot \inner{\int \, R \, dX }{\psi_{k}}  \\
& = & \sum_{k=1}^{n} \, h_{k} \cdot \int \, R' \psi_{k} \, dX 
= \sum_{k=1}^{n}  \int \, h_{k}  R' \psi_{k} \, dX = \int \, R'(H) \, dX. 
\end{eqnarray*}  
This proves \eqref{eqAssociaStochIntegral} for elementary $H$. Our objective is then to show $Y$ is a $S^{0}$-good integrator. That is, we must prove that the stochastic integral mapping $J:b\mathcal{P}(\Psi)\rightarrow S^{0}$, $J(H)=\int \, H \, dY$ is continuous. 

To show $J$ is continuous, it suffices to prove the existence of a continuous and linear operator $I:b\mathcal{P}(\Psi)\rightarrow S^{0}$ that coincides with $J$ on a dense subset of $b\mathcal{P}(\Psi)$. Let $(R_{\gamma}: \gamma \in \Gamma)$ be a net of elementary integrands converging to $R$ as in Lemma \ref{lemmaEquiNetSimpleBoundApproxPrope}. For every $\gamma \in \Gamma$ let $I_{\gamma}:b\mathcal{P}(\Psi)\rightarrow S^{0} $ given by $I_{\gamma}(H)=\int \, R_{\gamma}'(H) \, dX$. We must show $I_{\gamma}\in \mathcal{L}(b\mathcal{P}(\Psi),S^{0}) $. 

In effect, assume for the moment that $R_{\gamma} = \sum_{k=1}^{m} h_{k} A_{k}$ where $m \in \N$, $h_{k} \in b\mathcal{P}$ and $A_{k} \in \mathcal{L}(\Phi',\Psi')$ for $k=1, \cdots, m$ (for simplicity we omit the dependence on $\gamma$). Then, if $H_{\alpha} \rightarrow H$ in $b\mathcal{P}(\Psi)$ and if $p$ is a continuous seminorm on $\Phi$, we have 
\begin{eqnarray*}
\sup_{(t,\omega)} p \left( R'_{\gamma}(H_{\alpha}(t,\omega)-H(t,\omega))\right) 
& \leq  & \sup_{(t,\omega)} \sum_{k=1}^{m} \abs{h_{k}(t,\omega)} p \left( A'_{k} (H_{\alpha}(t,\omega)-H(t,\omega))\right) \\
 & \leq  &  \sum_{k=1}^{m}  \sup_{(t,\omega)}\abs{h_{k}(t,\omega)} \tilde{p}_{k} (H_{\alpha}(t,\omega)-H(t,\omega))
\rightarrow 0, 
\end{eqnarray*}
where in the above we have used the fact that each $\tilde{p}_{k}(\cdot)=p\left( A'_{k} \cdot \right)$ is a continuous seminorm on $\Psi$. Since $X$ is a $S^{0}$-good integrator we have 
$$ I_{\gamma}(H_{\alpha})=\int \, R'_{\gamma}(H_{\alpha}) \, dX 
\overset{S^{0}}{\rightarrow}  \int \, R'_{\gamma}(H) \, dX = I_{\gamma}(H). $$
Therefore $I_{\gamma} \in \mathcal{L}(b\mathcal{P}(\Psi),S^{0}) $. 


To prove the existence of an operator $I$ as described above we will need to show that $(I_{\gamma}: \gamma \in \Gamma)$  is an equicontinuous subset of $\mathcal{L}(b\mathcal{P}(\Psi),S^{0})$  and that $I_{\gamma}(H) \rightarrow J(H)$ in $S^{0}$ for each elementary $H \in b\mathcal{P}(\Psi) $. 

To prove $(I_{\gamma}: \gamma \in \Gamma)$  is equicontinuous notice that for every $\gamma \in \Gamma$ we have $I_{\gamma}=K \circ \widehat{R}'_{\gamma}$, where $\widehat{R}'_{\gamma} \in \mathcal{L}(b\mathcal{P}(\Psi),b\mathcal{P}(\Phi))$ is the mapping $H \mapsto R'(H)$  and $K: b\mathcal{P}(\Psi) \rightarrow S^{0}$ is the stochastic integral mapping of $X$, i.e. $K(H)=\int \, H \, dX$ for every $H \in b\mathcal{P}(\Psi)$. 

The family $(\widehat{R}'_{\gamma}:\gamma \in \Gamma)$ is equicontinuous. To see why this is true, observe that since the family $(R'_{\gamma}:\gamma \in \Gamma)$ is equicontinuous (Lemma \ref{lemmaEquiNetSimpleBoundApproxPrope}), given a continuous seminorm $p$ on $\Phi$ there exists a continuous seminorm $q$ on $\Psi$ such that 
$$ \sup_{(t,\omega)} p(R'_{\gamma}(t,\omega)\psi) \leq q(\psi), \quad \forall \psi \in \Psi, \, \gamma \in \Gamma. $$  
From the above we clearly have 
$$ \sup_{(t,\omega)} p(R'_{\gamma}(t,\omega)H(t,\omega)) \leq q(H(t,\omega)), \quad \forall H \in b\mathcal{P}(\Psi), \, \gamma \in \Gamma. $$  
which shows that $(\widehat{R}'_{\gamma}:\gamma \in \Gamma)$ is equicontinuous in $\mathcal{L}(b\mathcal{P}(\Psi),b\mathcal{P}(\Phi))$. 
Furthermore, since $K$ is continuous (recall $X$ is a $S^{0}$-good integrator), we conclude that the family $(I_{\gamma}: \gamma \in \Gamma)$  is equicontinuous. 

Now, given $H \in b\mathcal{P}(\Psi)$  there exists a bounded subset $B \subseteq \Psi$ such that $H(t,\omega) \in B$ $\forall t \geq 0$, $\omega \in \Omega$. Since $R_{\gamma} \rightarrow R$ under the topology of bounded convergence (Lemma \ref{lemmaEquiNetSimpleBoundApproxPrope}), for any continuous seminorm $p$ on $\Phi$ we have
\begin{multline*}
\sup_{(t,\omega)} p \left( (R'_{\gamma} (t,\omega)-R'(t,\omega) )H(t,\omega) \right) 
\leq \sup_{\psi \in B} \sup_{(t,\omega)} p \left( R'_{\gamma} (t,\omega)\psi -R'(t,\omega) \psi \right) \rightarrow 0.  
\end{multline*} 
That is $R_{\gamma}'(H) \rightarrow R'(H)$ in $b\mathcal{P}(\Phi)$. Since $X$ is a $S^{0}$-good integrator we have $I_{\gamma}(H) \rightarrow \widehat{I}(H)$ in $S^{0}$  for $\widehat{I}(H)= \int \, R'(H) \, dX$. 
Recall that for elementary $H \in b\mathcal{P}(\Psi)$ we have already proved that $\widehat{I}(H)=J(H)$ (this is \eqref{eqAssociaStochIntegral}), thus $I_{\gamma}(H) \rightarrow J(H)$. Because the $\Phi$-valued elementary processes are dense in $ b\mathcal{P}(\Psi)$, by Corollary 1 in Section 34.3 in \cite{Treves} there exists  $I\in \mathcal{L}(b\mathcal{P}(\Psi),S^{0})$ that coincides with $J$ for every $\Phi$-valued elementary processes and moreover $I_{\gamma}$ converges to $I$ for the topology of compact convergence (which is equivalent to the topology of bounded convergence because the space $\Psi$ is Montel). This proves $Y$ is a $S^{0}$-good integrator. 

Finally,  since both $Y$ and $X$ are $S^{0}$-good integrators in $\Phi'$ and the collection of all the elementary processes is dense in $ b\mathcal{P}(\Psi)$, then \eqref{eqAssociaStochIntegral} extends to every $H \in b\mathcal{P}(\Psi)$ by continuity. 
\end{proof}

It is known (see \cite{Jarchow}, Section 14.8, Example 5(d), p.319) that $\mathscr{S}(\R)$ possesses a Schauder basis, hence have the bounded approximation property (Remark \ref{remaBoundedApproxima}). Then, by Theorem \ref{theoIntegralGoodIntegrator} we obtain the following useful conclusion:

\begin{corollary}\label{coroStocIntegTemperedIsGoodIntegrator}
Let  $X=(X_{t}: t \geq 0)$ be a $S^{0}$-good integrator in $\mathscr{S}'(\R)$ and $R \in  \mathcal{P}_{loc}(\mathscr{S}'(\R),\mathscr{S}'(\R))$. Then $Y=\int \, R \, dX$ is a $S^{0}$-good integrator in $\mathscr{S}'(\R)$.
\end{corollary}

%


\section{It\^{o}'s formula in the space of distributions}\label{sectItoFormula}

In \cite{Ustunel:1982Ito}, A. S.  \"{U}st\"{u}nel introduced a generalization of It\^{o}'s formula to the space of (tempered) distributions. This is carried out by replacing functions of a given semimartingale by  the convolution of a distribution with the Dirac measure corresponding to the point in space at which evaluation of the semimartingale is desired.

To be precise, recall that if $T$ is a distribution and $z=(z_{t}: t \geq 0)$ is an $\R^{d}$-valued semimartingale, then for every test function $\phi$ we have (see \cite{Treves}, Proposition 27.6, p.296):
$$ \inner{T \ast \delta_{z_{t}}}{\phi} = T(\phi(\cdot +z_{t})).$$
Thus, in \cite{Ustunel:1982Ito} an It\^{o}'s formula is calculated for $T \ast \delta_{z_{t}}$.  

In this section we obtain another proof of \"{U}st\"{u}nel's version of It\^{o}'s formula from \cite{Ustunel:1982Ito} by applying the $S^{0}$-good integrator property.  

\begin{theorem}\label{theoItoFormula}
Let $z=(z_{t}: t \geq 0)$ be a $\R^{d}$-valued semimartingale. For every $T \in \mathscr{S}'(\R^{d})$ we have:
\begin{enumerate}
\item \label{itemConvDelta} $T \ast \delta_{z_{t}}$, $\nabla T \ast \delta_{z_{t}}$ and $\Delta T \ast \delta_{z_{t}}$ are all $S^{0}$-good integrators in $\mathscr{S}'(\R^{d})$. 
\item \label{itemStocIntegItoForm} The processes 
\begin{align*}
A_{t} &=  \int_{0}^{t} \,  \nabla T \ast \delta_{z_{s-}}  \, dz_{s},\\
B_{t} &=  \int_{0}^{t} \, \Delta T \ast \delta_{z_{s-}} \, d \inner{z^{c}}{z^{c}}_{s},\\
C_{t} &= \sum_{s \leq t} \left[ T \ast \delta_{z_{s}} - T \ast \delta_{z_{s-}} + \nabla T \ast \delta_{z_{s-}} \Delta z_{s} \right],
\end{align*}
are all $\mathscr{S}'(\R^{d})$-valued regular c\`{a}dl\`{a}g $(\mathcal{F}_{t})$-semimartingales. In particular,  $(C_{t}: t \geq 0) $ is a process of finite variation.
\item \label{itemItoForm} We have the following relation with equality in $S^{0}(\mathscr{S}'(\R^{d}))$:
\begin{eqnarray}
T \ast \delta_{z_{t}} 
& = & T \ast \delta_{z_{0}} -  \int_{0}^{t} \,  \nabla T \ast \delta_{z_{s-}}  \, dz_{s} + \frac{1}{2} \int_{0}^{t} \, \Delta T \ast \delta_{z_{s-}} \, d \llangle  z^{c}, z^{c} \rrangle _{s} \nonumber \\
& {} &  + \sum_{s \leq t} \left[ T \ast \delta_{z_{s}} - T \ast \delta_{z_{s-}} + \nabla T \ast \delta_{z_{s-}} \Delta z_{s} \right] \label{eqItoFormulaSpaceDistributions}.
\end{eqnarray}
\end{enumerate}
Moreover, the above conclusions remain valid if we replace the space $\mathscr{S}'(\R^{d})$ with the space $\mathscr{D}'(\R^{d})$. 
\end{theorem}

\begin{remark}
Conclusions \ref{itemStocIntegItoForm} and \ref{itemItoForm} in Theorem \ref{theoItoFormula} were obtained by \"{U}st\"{u}nel in \cite{Ustunel:1982Ito} using a different approach of stochastic integration. Our main contribution in Theorem \ref{theoItoFormula} is  
\ref{itemConvDelta} which in particular enlarges our collection of examples of $S^{0}$-good integrators. 
\end{remark}

\begin{proof}[Proof of Theorem \ref{theoItoFormula}]
Let $z=(z_{t}: t \geq 0)$ be a $\R^{d}$-valued semimartingale and let  $T \in \mathscr{S}'(\R^{d})$. 

We prove \ref{itemConvDelta}. 
By Example \ref{examDiracGoodIntegrator} we have $\delta_{z_{t}}$ is a $S^{0}$-good integrator in $\mathscr{E}'(\R^{d})$. It is well known that the mapping $S \mapsto T \ast S$ is linear continuous from $\mathscr{E}'(\R^{d})$ into 
$\mathscr{S}'(\R^{d})$ (see Theorem 30.1 in \cite{Treves}, p.316). Then by Proposition \ref{propContinuousImageGoodIntegrator} we have $ T \ast \delta_{z_{t}}$ is a $S^{0}$-good integrator in $\mathscr{S}'(\R^{d})$.  
Likewise, since $\nabla T \in \mathscr{S}'(R^{d})$ and $\nabla T  \in \mathscr{S}'(R^{d})$, again by Example \ref{examDiracGoodIntegrator} and Proposition \ref{propContinuousImageGoodIntegrator} we conclude that  $\nabla T \ast \delta_{z_{t}}$ and $\Delta T \ast \delta_{z_{t}}$ are $S^{0}$-good integrators in $\mathscr{S}'(\R^{d})$. 

Now we prove \ref{itemStocIntegItoForm}. For given $i=1, \dots ,d$ denote by $z^{i}$ the i-th component of $z$, i.e. $z_{t}=(z^{1}_{t}, \dots, z^{d}_{t})$. Also, denote $\partial/\partial x_{i}$ by $\partial_{i}$ and $\partial^{2} /\partial x_{j} \partial x_{i}$ by $\partial^{2}_{ij}$ for $i,j=1, \dots,  d$. 

Since  $\nabla T \ast \delta_{z_{t}}$ and $\Delta T \ast \delta_{z_{t}}$ are $S^{0}$-good integrators in $\mathscr{S}'(\R^{d})$, 
by Corollary 5.4 in \cite{FonsecaMora:StochInteg} we have $(\nabla T \ast \delta_{z_{t-}}: t \geq 0)$ and 
$(\Delta T \ast \delta_{z_{t-}} : t \geq 0) $ are elements of $\mathcal{P}_{loc}(\mathscr{S}'(\R^{d}))$. Therefore, by Example \ref{examVectorStochIntegRealValueSemimar} the stochastic integrals $\int_{0}^{t} \,  \nabla T \ast \delta_{z_{r-}}  \, dz_{r}$ and $ \int_{0}^{t} \, \Delta T \ast \delta_{z_{r-}} \, d \llangle  z^{c}, z^{c} \rrangle _{r}$ are $\mathscr{S}'(\R^{d})$-valued regular  c\`{a}dl\`{a}g $(\mathcal{F}_{t})$-semimartingales. Moreover, by \eqref{eqWeakStrongCompaIntegralRealSemima} we have $\Prob$-a.e. for every $\phi \in \mathscr{S}(\R^{d})$ and $t \geq 0$ that
\begin{eqnarray}
\inner{\int_{0}^{t} \,  \nabla T \ast \delta_{z_{s-}}  \, dz_{s} }{\phi} & = & \sum_{i=1}^{n} \int_{0}^{t} \,  \inner{ \partial_{i} T \ast \delta_{z_{s-}}}{\phi}  \, dz^{i}_{s} \nonumber \\
& = & - \sum_{i=1}^{n} \int_{0}^{t} \,   T \ast \delta_{z_{s-}}(\partial_{i} \phi)  \, dz^{i}_{s}  \nonumber \\
& = & - \sum_{i=1}^{n} T \left( \int_{0}^{t} \,  \partial_{i} \phi(\cdot + z_{s-})  \, dz^{i}_{s} \right),
\label{eqWeakStochIntegItoFormula}
\end{eqnarray}
and 
\begin{eqnarray}
 \inner{\int_{0}^{t} \,  \Delta T \ast \delta_{z_{s-}}  \, dz_{s} }{\phi} & = & \sum_{i,j=1}^{n} \int_{0}^{t} \,  \inner{ \partial^{2}_{ij}  T \ast \delta_{z_{s-}}}{\phi}  \,  d \llangle  z_{i}^{c}, z_{j}^{c} \rrangle _{s} \nonumber \\
 & = & \sum_{i,j=1}^{n} \int_{0}^{t} \,   T \ast \delta_{z_{s-}}(\partial^{2}_{ij} \phi)  \, d \llangle  z_{i}^{c}, z_{j}^{c} \rrangle_{s} \nonumber \\
& = & \sum_{i,j=1}^{n} T \left( \int_{0}^{t} \,   \partial^{2}_{ij} \phi ( \cdot + z_{s-})  \, d \llangle  z_{i}^{c}, z_{j}^{c} \rrangle_{s} \right). \label{eqWeakIntegCovariItoFormula}
\end{eqnarray}

Let $\Omega_{z} \subseteq \Omega$ with $\Prob(\Omega_{z})=1$ such that $\sum_{s \leq t} \norm{\Delta z_{s}}^{2}< \infty$, $\forall t >0, \omega \in \Omega_{z}$. Observe that if $\omega \in \Omega_{z}$, there is at most countably many jumps of $z(\omega)$ on $[0,t]$. We must show that $(C_{t}: t \geq 0)$ is a $\mathscr{S}'(\R^{d})$-valued process. 

First, observe that for any $\phi \in \mathscr{S}(\R^{d})$ we have
\begin{equation}\label{eqdualityCtProofItoFormula}
\inner{C_{t}}{\phi}  =  \sum_{s \leq t} T \left( \phi( \cdot + z_{s}) - \phi( \cdot + z_{s-}) +\sum_{i=1}^{d}  \partial_{i} \phi( \cdot + z_{s-}) \Delta z^{i}_{s} \right).
\end{equation}
Let 
$$\alpha_{t}(\phi(\cdot)) \defeq  \sum_{s \leq t} \left[ \phi( \cdot + z_{s}) - \phi( \cdot + z_{s-}) +\sum_{i=1}^{d}  \partial_{i} \phi( \cdot + z_{s-}) \Delta z^{i}_{s} \right] $$
By Taylor's theorem, for every continuous  seminorm $q$ on $\mathscr{S}(\R^{d})$ we have
$$ q( \alpha_{t}(\phi(\cdot))) \leq \frac{1}{2} \left(  \sum_{i,j=1}^{d} \sup_{s \leq t} q \left( \partial^{2}_{ij} \phi(\cdot+z_{s-} +\theta \Delta z_{s} )\right) \right)\sum_{s \leq t} \norm{\Delta z_{s}}^{2},  \quad \theta \in (0,1).  $$
For $\omega \in \Omega_{z}$, the set $K_{\omega,t}=\overline{\mbox{co}} \left( \{ z_{s}(\omega): s \leq t\} + \{ z_{s-}(\omega): s \leq t\} \right)$ is compact, hence
$$ q( \alpha_{t}(\phi(\cdot))(\omega)) \leq \frac{1}{2} \left(  \sum_{i,j=1}^{d} \sup_{y \in K_{\omega,t}} q \left( \partial^{2}_{ij} \phi(\cdot+y)\right) \right)\sum_{s \leq t} \norm{\Delta z_{s}(\omega)}^{2}. $$
Since for $T \in \mathscr{S}'(\R^{d})$ there exists some continuous seminorm $p$ on $\mathscr{S}(\R^{d})$ such that $T \in \mathscr{S}'(\R^{d})_{p}$, the above estimate (replacing $q$ with $p$) together with \eqref{eqdualityCtProofItoFormula} shows that  for any  $\omega \in \Omega_{z}$ and $t \geq 0$ we have 
\begin{equation*}
\abs{\inner{C_{t}(\omega)}{\phi}} \leq N_{\omega,t,T}  \varrho(\phi), \quad \forall \, \phi \in \mathscr{S}(\R^{d}), 
\end{equation*}
where $N_{\omega,t,T}= \frac{ p'(T)}{2} \sum_{s \leq t} \norm{\Delta z_{s}(\omega)}^{2} < \infty$ and 
\begin{equation}\label{eqDefiSeminormItoFormula} 
\varrho(\phi) \defeq \sum_{i,j=1}^{d} \sup_{y \in K_{\omega,t}} p \left( \partial^{2}_{ij} \phi(\cdot+y)\right), \quad \forall \, \phi \in \mathscr{S}(\R^{d}),
\end{equation}
is a continuous seminorm on $\mathscr{S}(\R^{d})$. Then by \eqref{eqDefiSeminormItoFormula} we conclude  $C_{t}(\omega) \in \mathscr{S}'(\R^{d})$ for any  $\omega \in \Omega_{z}$ and $t \geq 0$. Redefining $C_{t}(\omega)=0$ for $\omega \notin \Omega_{z}$ and $t \geq 0$, we conclude that $(C_{t}: t \geq 0)$    is a $ \mathscr{S}'(\R^{d})$-valued regular $(\mathcal{F}_{t})$-adapted process. Moreover, $(C_{t}: t \geq 0)$ is a $ \mathscr{S}'(\R^{d})$-valued semimartingale of finite variation. 

In effect, given $\phi \in \mathscr{S}(\R^{d})$, by  \eqref{eqdualityCtProofItoFormula} and similar arguments to those used above, for $\omega \in \Omega_{z}$ and $t >0$, if we choose any partition $0=t_{0} < t_{1} < \cdots < t_{m}=t$ of $[0,t]$, we have 
\begin{flalign*}
& \sum_{i=1}^{m} \abs{\inner{C_{t_{i}}}{\phi}-\inner{C_{t_{i-1}}}{\phi}} \\
& \leq  \sum_{s \leq t}  \abs{T \left( \phi( \cdot + z_{s}) - \phi( \cdot + z_{s-}) +\sum_{i=1}^{d}  \partial_{i} \phi( \cdot + z_{s-}) \Delta z^{i}_{s} \right)}  \leq N_{\omega,t,T} \varrho (\phi) < \infty. 
\end{flalign*}
Since the last line of the above inequality is independent of the choice of the partition, we have $(\inner{C_{t}}{\phi}: t \geq 0) $ is a process of finite variation. Now, since $(C_{t}: t \geq 0)$ is a regular process and $\mathscr{S}(\R^{d})$ is barrelled, by Theorem 2.10 in \cite{FonsecaMora:Existence} the distribution of each $C_{t}$ is a Radon probability measure. Then, by the regularization theorem for semimartingales (Proposition 3.12 in \cite{FonsecaMora:Semi}) we have  $(C_{t}: t \geq 0) $ has an indistinguishable  $ \mathscr{S}'(\R^{d})$-valued $(\mathcal{F}_{t})$-adapted, regular, c\`{a}dl\`{a}g version. 

To prove \ref{itemItoForm}, let  $\phi \in \mathscr{S}(\R^{d})$ and by It\^{o}'s formula for finite dimensional semimartingales we have
\begin{eqnarray*}
\phi(\cdot + z_{t}) & = & \phi(\cdot + z_{0}) + \sum_{i=1}^{d} \int_{0}^{t} \, \partial_{i} \phi(\cdot + z_{s-}) \, dz^{i}_{s}  + \frac{1}{2} \sum_{i,j=1}^{d}  \int_{0}^{t} \, \partial^{2}_{ij} \phi(\cdot + z_{s-})  \, d \llangle  z_{i}^{c}, z_{j}^{c} \rrangle _{s} \\
& {} & + \sum_{s \leq t} \left[ \phi( \cdot + z_{s}) - \phi( \cdot + z_{s-}) +\sum_{i=1}^{d}  \partial_{i} \phi( \cdot + z_{s-}) \Delta z^{i}_{s} \right]
\end{eqnarray*}
Then, if we apply $T$ to both sides of the above equality and use 
\eqref{eqWeakStochIntegItoFormula}, \eqref{eqWeakIntegCovariItoFormula} and \eqref{eqdualityCtProofItoFormula} we obtain 
\begin{eqnarray*}
\inner{T \ast \delta_{z_{t}}}{\phi} 
& = & \left\langle T \ast \delta_{z_{0}} -  \int_{0}^{t} \,  \nabla T \ast \delta_{z_{s-}}  \, dz_{s} + \frac{1}{2} \int_{0}^{t} \, \Delta T \ast \delta_{z_{s-}} \, d \llangle  z^{c}, z^{c} \rrangle _{s}  \right. \\
& {} & \left. + \sum_{s \leq t} \left[ T \ast \delta_{z_{s}} - T \ast \delta_{z_{s-}} + \nabla T \ast \delta_{z_{s-}} \Delta z_{s} \right] \, , \phi \, \right\rangle. 
\end{eqnarray*}
Since by \ref{itemConvDelta} and \ref{itemStocIntegItoForm} all the processes in the above equality are  $\mathscr{S}'(\R^{d})$-valued regular c\`{a}dl\`{a}g $(\mathcal{F}_{t})$-semimartingales, then by Proposition 2.12 in \cite{FonsecaMora:Existence} we obtain \eqref{eqItoFormulaSpaceDistributions}.

The above proof can be modified so that all the conclusions remain valid for the space $\mathscr{D}'(\R^{d})$. In effect, let $T \in \mathscr{D}'(\R^{d})$. Since the mapping $S \mapsto T \ast S$ is linear continuous from $\mathscr{E}'(\R^{d})$ into 
$\mathscr{D}'(\R^{d})$ (see Theorem 27.6 in \cite{Treves}, p.294), as before  by Example \ref{examDiracGoodIntegrator} and Proposition \ref{propContinuousImageGoodIntegrator} we conclude that $ T \ast \delta_{z_{t}}$,  $\nabla T \ast \delta_{z_{t}}$ and $\Delta T \ast \delta_{z_{t}}$ are $S^{0}$-good integrators in $\mathscr{D}'(\R^{d})$. Likewise, by Example \ref{examVectorStochIntegRealValueSemimar} and 
Corollary 5.4 in \cite{FonsecaMora:StochInteg} we can show $\int_{0}^{t} \,  \nabla T \ast \delta_{z_{r-}}  \, dz_{r}$ and $ \int_{0}^{t} \, \Delta T \ast \delta_{z_{r-}} \, d \llangle  z^{c}, z^{c} \rrangle _{r}$ are $\mathscr{D}'(\R^{d})$-valued regular c\`{a}dl\`{a}g $(\mathcal{F}_{t})$-semimartingales and that \eqref{eqWeakStochIntegItoFormula}, \eqref{eqWeakIntegCovariItoFormula} remain valid. Moreover, the proof that $(C_{t}: t \geq 0)$    is a $ \mathscr{D}'(\R^{d})$-valued regular $(\mathcal{F}_{t})$-semimartingale of finite variation remain the same since for every continous seminorm $p$ on $ \mathscr{D}(\R^{d})$ we have  \eqref{eqDefiSeminormItoFormula} defines a continuous seminorm on $ \mathscr{D}(\R^{d})$. 

Finally, now that we have  \ref{itemConvDelta} and \ref{itemStocIntegItoForm} are valid for the space $\mathscr{D}'(\R^{d})$, the proof of \ref{itemItoForm} remains the same for $\mathscr{D}'(\R^{d})$-valued semimartingales. 
\end{proof}

\textbf{Acknowledgements} {  This work was supported by The University of Costa Rica through the grant ``821-C2-132- Procesos cil\'{i}ndricos y ecuaciones diferenciales estoc\'{a}sticas''.}

\end{document}